\documentclass[12pt,a4paper]{amsart}

\usepackage{amsmath,amsfonts,amssymb,mathtools,extarrows}
\usepackage{color}

\usepackage[hmargin=2cm,vmargin=2cm]{geometry}

\usepackage{hyperref}
\usepackage{nameref,zref-xr}     

\usepackage{enumerate}      

\usepackage{scalerel}         

\newcommand{\mbP}{\mathbb P}
\newcommand{\mbZ}{\mathbb Z}
\newcommand{\mbC}{\mathbb C}
\newcommand{\mbR}{\mathbb R}

\newcommand{\ord}{\operatorname{ord}}
\newcommand{\oM}{\overline{\mathcal M}}

\newcommand{\og}{\overline g}

\newcommand{\hLambda}{\widehat\Lambda}

\def\cM{{\mathcal{M}}}
\def\oM{{\overline{\mathcal{M}}}}
\def\CP{{{\mathbb C}{\mathbb P}}}
\renewcommand{\Im}{\mathrm{Im}}

\def\d{{\partial}}

\newcommand{\eps}{\varepsilon}

\newcommand{\cA}{\mathcal A}
\newcommand{\hcA}{\widehat{\mathcal A}}
\newcommand{\DR}{\mathrm{DR}}

\newcommand{\even}{\mathrm{even}}
\newcommand{\ct}{\mathrm{ct}}

\newcommand{\Coef}{\mathrm{Coef}}

\DeclareMathOperator{\Deg}{Deg}

\newcommand{\gl}{\mathrm{gl}}

\newcommand{\tR}{\widetilde{R}}

\newcommand{\tQ}{\widetilde{Q}}

\newcommand{\hE}{\widehat{E}}

\newcommand{\un}{{1\!\! 1}}

\newcommand{\of}{\overline{f}}

\newcommand{\cR}{\mathcal{R}}

\newcommand{\pt}{\mathrm{pt}}

\DeclareMathOperator{\res}{{res}}

\newcommand{\tS}{\widetilde{S}}

\newcommand{\oB}{\overline{B}}
\newcommand{\codim}{\operatorname{codim}}
\newcommand{\spn}{\operatorname{span}}
\DeclareMathOperator{\tdeg}{\widetilde{\mathrm{deg}}}
\newcommand{\hcR}{\widehat{\mathcal{R}}}
\newcommand{\oH}{\overline{\mathcal{H}}}
\newcommand{\cH}{\mathcal{H}}
\newcommand{\ev}{\mathrm{ev}}
\newcommand{\tP}{\widetilde{P}}
\newcommand{\cB}{\mathcal{B}}
\DeclareMathOperator{\odeg}{\overline{\mathrm{deg}}}
\newcommand{\tf}{\widetilde{f}}
\newcommand{\tL}{\widetilde{L}}
\newcommand{\tw}{\widetilde{w}}
\newcommand{\mbCP}{\mathbb{C}\mathbb{P}}
\newcommand{\adm}{\mathrm{adm}}

\newtheorem{theorem}{Theorem}[section]
\newtheorem{proposition}[theorem]{Proposition}
\newtheorem{lemma}[theorem]{Lemma}

\theoremstyle{remark}

\newtheorem{remark}[theorem]{Remark}
\newtheorem{example}[theorem]{Example}

\theoremstyle{definition}

\newtheorem{definition}[theorem]{Definition}

\usepackage{color}

\numberwithin{equation}{section}

\begin{document}

\title[Residueless meromorphic differentials 
and KP]{Moduli spaces of residueless meromorphic differentials 
and the KP hierarchy}

\author{Alexandr Buryak}
\address{A. Buryak:\newline 
Faculty of Mathematics, National Research University Higher School of Economics, \newline
6 Usacheva str., Moscow, 119048, Russian Federation;\smallskip\newline 
Center for Advanced Studies, Skolkovo Institute of Science and Technology, \newline
1 Nobel str., Moscow, 143026, Russian Federation;\smallskip\newline
P.G. Demidov Yaroslavl State University,\newline 
14 Sovetskaya str., Yaroslavl, 150003, Russian Federation}
\email{aburyak@hse.ru}

\author{Paolo Rossi}
\address{P. Rossi:\newline Dipartimento di Matematica ``Tullio Levi-Civita'', Universit\`a degli Studi di Padova,\newline
Via Trieste 63, 35121 Padova, Italy}
\email{paolo.rossi@math.unipd.it}

\author{Dimitri Zvonkine}
\address{Dimitri Zvonkine: \newline Universit\'e de Versailles St-Quentin, CNRS,\newline
45 Avenue des \'Etats Unis, 78000 Versailles, France}
\email{dimitri.zvonkine@gmail.fr}

\begin{abstract}
We prove that the cohomology 
classes of the moduli spaces of residueless meromorphic differentials, 
i.e., the closures, in the moduli space of stable curves, of the loci of smooth curves whose marked points are the zeros and poles of prescribed orders of a meromorphic differential with vanishing residues, form a partial cohomological field theory (CohFT) of infinite rank. To this partial CohFT we apply the double ramification hierarchy construction to produce a Hamiltonian system of evolutionary PDEs. We prove that its reduction 
to the case of differentials with exactly two zeros and any number of poles coincides with the KP hierarchy up to a change of variables. 
\end{abstract}

\date{\today}

\maketitle

\tableofcontents

\section*{Introduction}

In recent years several constructions of moduli spaces of meromorphic differentials on smooth Riemann surfaces, where both the differential and the curve are allowed to vary, have appeared in the literature. In particular, in \cite{BCGGM18,BCGGM19,Sau19} the authors constructed, with different techniques, smooth Deligne--Mumford moduli stacks parameterizing families of stable curves of genus~$g$ and with~$n$ markings, together with a meromorphic differential with poles and zeros of prescribed orders $a_1,\ldots,a_n \in \mbZ$, $\sum_{i=1}^n a_i=2g-2$, on their $n$ marked points and studied their geometry and topology. Such families have a natural univocal definition as long as the underlying curve is smooth, in which case their moduli stack, up to projectivization with respect to the multiplicative $\mbC^*$-action on the differential, can be seen as a substack $\cH_g(a_1,\ldots,a_n)$ inside $\cM_{g,n}$. The above constructions provide different compactifications and all possess natural forgetful maps to the moduli space of stable curves $\oM_{g,n}$ with respect to which their image is simply the closure $\oH_g(a_1,\ldots,a_n)$,  which is pure dimensional, but is not in general irreducible. This is in contrast, for instance, with \cite{FP18} where the authors construct a closed  pure dimensional  substack $\widetilde{\cH}_g(a_1,\ldots,a_n)$ of $\oM_{g,n}$ as a proper moduli space of twisted canonical divisors containing  $\cH_g(a_1,\ldots,a_n)$ as an open subset, but having in general irreducible components that do not lie in $\oH_g(a_1,\ldots,a_n)$.  In the strictly meromorphic case, where there exists an $a_i<0$, the moduli space $\widetilde{\cH}_g(a_1,\ldots,a_n)$  carries a natural weighted fundamental class $\mathsf{H}_{g}(a_1\ldots,a_n)$ which was shown in \cite{BHPSS20} to equal Pixton's $1$-twisted double ramification (DR) cycle $\DR_g^1(a_1,\ldots,a_n)$, defined in \cite{JPPZ17} as an explicit sum over stable graphs of tautological classes.\\

While Pixton's formula is expected to provide the weighted fundamental classes $\mathsf{H}_{g}(a_1\ldots,a_n)$ with the structure of an infinite rank partial cohomological field theory (CohFT), as already proven for the (untwisted) DR cycle in \cite{BR21c} (see also \cite{BR21b}), we cannot expect the same from the fundamental classes of $\oH_g(a_1,\ldots,a_n)$, simply for dimensional reasons. The situation however improves if we demand that all residues of the meromorphic differentials vanish. The corresponding moduli stacks and compactifications were constructed in \cite{Sau19,CMZ20} and the corresponding substack of $\oM_{g,n}$ is denoted by $\oH_g^{\res}(a_1,\ldots,a_n)$.\\

Our first result is 
that the fundamental classes of $\oH_g^{\res}(a_1,\ldots,a_n)$ (with $a_i\neq -1$ for all $1\leq i\leq n$) do indeed form an infinte rank partial CohFT. We show this in Section~\ref{section:Moduli spaces}, after introducing the necessary geometric notions and results from the aforementioned papers.\\

At this point the possibility of employing integrable systems techniques to study the intersection theory of $\oH_g^{\res}(a_1,\ldots,a_n)$ arises. In Section~\ref{section:DR hierarchy} we define the corresponding DR hierarchy and prove some of its properties, including homogeneity with respect to the appropriate grading.\\

Finally, our main result is found in Section~\ref{section:KP}, where we prove that a reduction of the DR hierarchy corresponding to moduli spaces of meromorphic differentials with exactly two zeros and any number of poles with no residues coincides with the celebrated Kadomtsev--Petviashvili (KP) hierarchy up to a Miura transformation.\\  

The precise indentification of the aforementioned reduction of the DR hierarchy for residueless meromorphic differentials with the KP hierarchy constructed via Lax operators is achieved thanks to a reconstruction theorem, also proved in Section~\ref{section:KP} of independent interest: the KP hierarchy can be uniquely reconstructed, using the properties of commutativity of the flows, homogeneity, tau-symmetry and compatibility with spatial translations, from exactly three coefficients in each component of the first nontrivial flow together with the linear terms in the dispersionless limit of all other flows.\\

Natural future developments include the identification of the full DR hierarchy for the spaces of residueless meromorphic differentials and the investigation of the Dubrovin--Zhang~\cite{DZ01} side of the correspondence of this partial cohomological field theory with integrable systems, guided by the DR/DZ equivalence conjecture \cite{Bur15,BDGR18} which predicts that the KP hierarchy and its parent hierarchy for differentials with any number of zeros should compute all intersection numbers of $\oH_g^{\res}(a_1,\ldots,a_n)$ with any monomial in the psi classes. This is material for a future work.\\

\noindent{\bf Notation and conventions.} 
\begin{itemize}

\item Throughout the text we use the Einstein summation convention for repeated upper and lower Greek indices.

\medskip

\item When it doesn't lead to a confusion, we use the symbol $*$ to indicate any value, in the appropriate range, of a sub- or superscript.

\medskip

\item For a topological space~$X$ let $H^*(X)$ denote the cohomology ring of~$X$ with the coefficients in $\mbC$.

\medskip

\item For $n\geq 0$, let $[n]:=\{1,\ldots,n\}$.

\end{itemize}

\bigskip

\noindent{\bf Acknowledgements.} 
The work of A.~B. is supported by the Russian Science Foundation (Grant no. 20-71-10110). D.~Z. is partly supported by the ANR-18-CE40-0009 ENUMGEOM grant.\\

We would like to thank Matteo Costantini, Adrien Sauvaget and Johannes Schmitt for useful discussions and guidance on the literature on moduli spaces of meromorphic differentials. We would also like to thank Michael Finkelberg for valuable comments related to the proof of Proposition~\ref{proposition:partial CohFT Hres}.\\


\section{Moduli spaces of meromorphic differentials with residue conditions}\label{section:Moduli spaces}

For two nonnegative integers $g,n$ such that $2g-2+n>0$, let $\oM_{g,n}$ be the moduli space of stable curves of genus $g$ with $n$ marked points, $\cM_{g,n}$ its open locus of smooth curves, and~$\cM^\ct_{g,n}$ the partial compactification of $\cM_{g,n}$ by curves of compact type, i.e. stable curves whose dual stable graph is a tree. Naturally $\cM_{g,n}\subset\cM^\ct_{g,n}\subset\oM_{g,n}$.\\

\subsection{Meromorphic differentials with residue conditions}

For integers $g,n,m,k\geq 0$ such that $2g-2+n+m+k>0$, fix integers $a_1,\ldots, a_n\geq 0$, $b_1,\ldots,b_m \geq 1$, $c_1,\ldots,c_k\geq 2$. The space of projectivized meromorphic differentials with vanishing residues at the last $k$ points is the subset
$$
\cH_{g}(a_1,\ldots,a_n,-b_1,\ldots,-b_m;-c_1,\ldots,-c_k)\subset \cM_{g,n+m+k}
$$
of smooth marked curves $[C;x_1,\ldots,x_{n+m+k}]$ on which there exists a meromorphic differential~$\omega$ whose associated divisor is $(\omega)=\sum_{j=1}^n a_j x_j-\sum_{j=1}^m b_j x_{n+j}-\sum_{j=1}^k c_j x_{n+m+j}$ and such that $\res_{x_{n+m+j}} \omega = 0$ for $1\leq j\leq k$. We denote its closure in $\oM_{g,n+m+k}$ by
$$
\oH_{g}(a_1,\ldots,a_n,-b_1,\ldots,-b_m;-c_1,\ldots,-c_k)\subset \oM_{g,n+m+k}.
$$
$\oH_g(a_1,\ldots,a_n,-b_1,\ldots,-b_m;-c_1,\ldots,-c_k)$ is a closed substack of $\oM_{g,n+m+k}$ of codimension $g+k$ if $m\geq 1$ and of codimension $g-1+k$ if $m=0$. It is empty unless the condition $\sum_{j=1}^n a_j - \sum_{j=1}^m b_j - \sum_{j=1}^k c_j= 2g-2$ is satisfied.\\

Notice that if $m=1$ and $[C;x_1,\ldots,x_{n+1+k}] \in \cH_g(a_1,\ldots,a_n,-b_1;-c_1,\ldots,-c_k)$, then the residue theorem implies that the meromorphic differential $\omega$ on $C$ satisfies $\res_{x_{n+1}}\omega = 0$ and hence $$\cH_g(a_1,\ldots,a_n,-b_1;-c_1,\ldots,-c_k)= \cH_g(a_1,\ldots,a_n;-b_1,-c_1,\ldots,-c_k),$$
so the case $m=1$ effectively reduces to $m=0$.\\

In the $k=0$ and $m=0$ cases, the notation can be simplified as follows.
\begin{definition}
Given $a_1,\ldots,a_n \in \mbZ$, let us introduce the following notation.
\begin{enumerate}[ 1.]
\item Denote by $\cH_g(a_1,\ldots,a_n)\subset\cM_{g,n}$ the space of projectivized meromorphic differentials, i.e. the locus in $\cM_{g,n}$ of smooth curves $[C;x_1,\ldots,x_n]$ on which there exists a meromorphic differential $\omega$ whose associated divisor is $(\omega)=\sum_{j=1}^n a_i x_i$. Denote moreover by $\oH_g (a_1,\ldots,a_n)$ its closure in $\oM_{g,n}$.

\medskip

\item Similarly, denote by $\cH_g^{\res} (a_1,\ldots,a_n)\subset \cM_{g,n}$ the space of projectivized meromorphic differentials with everywhere vanishing residues, i.e. the locus in $\cM_{g,n}$ of smooth curves $[C;x_1,\ldots,x_n]$ on which there exists a meromorphic differential $\omega$ whose associated divisor is $(\omega)=\sum_{j=1}^n a_i x_i$ and whose residues vanish at \emph{all} poles. Denote moreover by $\oH_g^{\res} (a_1,\ldots,a_n)$ its closure in~$\oM_{g,n}$.
\end{enumerate}
\end{definition} 

\bigskip

Notice that $\oH_g^{\res} (a_1,\ldots,a_n)$ is empty if $a_i = -1$ for some $1\leq i\leq n$ and unless $\sum_{i=1}^na_i=2g-2$. For an index set $I$ of finite cardinality $|I|\geq 0$ and an $|I|$-tuple of integers $a_I=(a_i)_{i\in I}\in \mbZ^{|I|}$, let $N_{a_I}:=|\{i\in I\, |\, a_i<0\}|$ be the number of negative entries of $a_I$. Then 
\begin{equation}\label{eq:dim Hres}
\codim \oH_g^{\res} (a_1,\ldots,a_n) = g-1 + N_{a_{[n]}}.
\end{equation}
We call the homology class $[\oH_g^{\res} (a_1,\ldots,a_n)]\in H_{2(2g-2-N_{a_{[n]}})}(\oM_{g,n})$ the \emph{cycle of residueless meromorphic differentials} and, by abuse of language, we will use the same name and notation for its Poincar\'e dual cohomology class $[\oH_g^{\res} (a_1,\ldots,a_n)] \in H^{2(g-1+N_{a_{[n]}})}(\oM_{g,n})$.

\bigskip

\begin{remark}
 In the strictly meromorphic case, a closed substack $\widetilde{\cH}_g(a_1,\ldots,a_n)\subset\oM_{g,n}$ containing $\oH_g(a_1,\ldots,a_n)$ was constructed in~\cite{FP18}  as a proper moduli space of twisted canonical divisors, carrying a natural weighted fundamental class $\mathsf{H}_{g}(a_1\ldots,a_n) \in H^{2g}(\oM_{g,n})$. As proven in \cite{BHPSS20}, $\mathsf{H}_{g}(a_1\ldots,a_n)$ equals Pixton's $1$-twisted double ramification (DR) cycle $\DR_g^1(a_1,\ldots,a_n)$, which is defined in \cite{JPPZ17} as an explicit sum over stable graphs of tautological classes.
\end{remark}

\bigskip

\subsection{Multiscale differentials with residue conditions}

Let us briefly review the definition and properties of the moduli space $\oH_g^{\res} (a_1,\ldots,a_n)$ from the point of view of multiscale differentials with residue conditions as treated in \cite{CMZ20}.\\


In \cite[Sections~3 and~4.1]{CMZ20} (see also~\cite[Section~2]{BCGGM19}) the authors identify the space $\cH_g^{\res} (a_1,\ldots,a_n)$ with the corresponding stratum $B^{\res}_g(a_1,\ldots,a_n)$ inside the projectivized twisted Hodge bundle
$$
\mbP\Bigg(\pi_*\omega\Bigg(-\sum_{i\in [n]\, |\, a_i<0} a_i x_i\Bigg)\Bigg),
$$
where $\omega$ is the relative dualizing sheaf of the universal curve over $\cM_{g,n}$, via its projection to~$\cM_{g,n}$. Then they construct a proper smooth Deligne--Mumford stack $\oB^{\res}_g(a_1,\ldots,a_n)$ containing $B^{\res}_g(a_1,\ldots,a_n)$ as an open dense substack whose complement is a normal crossing divisor. The stack $\oB^{\res}_g(a_1,\ldots,a_n)$ is a moduli stack for families of equivalence classes of multiscale differentials with residue conditions. Let us recall their definition.\\

In what follows, given a stable curve $C$ with associated stable graph $\Gamma_C$, we will denote its irreducible components by $C_v$ for $v\in V(\Gamma_C)$ and we will use the same notation for the marked points of $C$ and the corresponding legs of the associated stable graph $\Gamma_C$, for nodes of $C$ and the corresponding edges of $\Gamma_C$, and for branches of nodes on irreducible components $C_v$ of $C$ and the corresponding half-edges of $\Gamma_C$. Given a leg $x_i\in L(\Gamma_C)$ or a half-edge $h\in H(\Gamma)$, we denote by $v(x_i)$ or $v(h)$ the vertex to which they are attached.\\

Firstly, an \emph{enhanced level graph} is a stable graph $\Gamma$ of genus $g$ with a set $L(\Gamma)$ of $n$ marked legs together with:
\begin{enumerate}
\item a total preorder\footnote{A preorder relation $\leq$ is reflexive and transitive, but $x \leq y$ and $y \leq x$ do not necessarily imply $x=y$.} on the set $V(\Gamma)$ of vertices. We describe this preorder by a surjective level function $\ell\colon V(\Gamma)\to \{0,-1,\ldots,-L\}$. An edge is called \emph{horizontal} if it is attached to vertices on the same level and \emph{vertical} otherwise. 

\medskip

\item a function $\kappa\colon E(\Gamma)\to \mbZ_{\geq 0}$ assigning a nonnegative integer~$\kappa_e$ to each edge $e\in E(\Gamma)$, such that $\kappa_e=0$ if and only if $e$ is horizontal.
\end{enumerate}
For every level $0\leq j\leq -L$, let $C_{(j)}$ be the (possibly disconnected) stable curve obtained from~$C$ by removing all irreducible components whose level is not $j$ and let $C_{(>j)}$  be the (possibly disconnected) stable curve obtained from~$C$ by removing all irreducible components whose level is smaller than or equal to~$j$.\\

Secondly, given a meromorphic differential~$\omega$ on a smooth curve $C$ and a point $p\in C$, if~$\omega$ has order $\ord_p \omega = a\neq -1$ at $p$ then for a local coordinate $z$ in a neighborhood of $p$ such that $z(p)=0$ we have, locally, $\omega = (cz^a + O(z^{a+1}))dz$ for some $c\in \mbC^*$. Then the $k=|a+1|$ roots $\zeta$ such that $\zeta^{a+1}=c^{-1}$ determine $k$ projectivized vectors $\left.\zeta\frac{\d}{\d z}\right|_p\in T_pC/\mbR_{>0}$ (if $a\geq 0$) or $\left.-\zeta\frac{\d}{\d z}\right|_p\in T_pC/\mbR_{>0}$ (if $a<-1$) which are called \emph{outgoing} or \emph{incoming prongs} of~$\omega$, respectively. The set of outgoing (resp. incoming) prongs at $p$ is denoted by $P^\mathrm{out}_p$ (resp. $P^\mathrm{in}_p$).\\

Thirdly, a \emph{multiscale differential} of profile $(a_1,\ldots,a_n)\in \mbZ^n$, with $\sum_{i=1}^n a_i = 2g-2$, on a stable curve~$C$ of genus $g$ with $n$ marked points $x_1,\ldots,x_n$, with \emph{zero residues} at $x_1,\ldots,x_n\in C$ consists of:
\begin{enumerate}
\item a structure of enhanced level graph $(\Gamma_C, \ell, \kappa)$ on the dual graph $\Gamma_C$ of $C$ (where a node is said to be vertical or horizontal if the corresponding edge is);

\medskip

\item a collection of meromorphic differentials $\omega_v$, one on each irreducible component~$C_v$ of~$C$, $v\in V(\Gamma_C)$, holomorphic and non-vanishing outside of marked points and nodes, such that the following conditions are satisfied:
\begin{itemize}
\item[(i)] $\ord_{x_i}\omega_{v(x_i)} = a_i$, $1\leq i \leq n$.

\smallskip

\item[(ii)] $\res_{x_i} \omega_{v(x_i)}=0, 1\leq i \leq n$.

\smallskip

\item[(iii)] If $q_1\in C_{v_1}$ and $q_2 \in C_{v_2}$, $v_1,v_2\in V(\Gamma_C)$, form a node $e\in E(\Gamma_C)$, then
\begin{equation*}
\ord_{q_1}\omega_{v_1}+\ord_{q_2}\omega_{v_2}=-2.
\end{equation*}

\smallskip

\item[(iv)] If $q_1\in C_{v_1}$ and $q_2 \in C_{v_2}$, $v_1,v_2\in V(\Gamma_C)$, form a node $e\in E(\Gamma_C)$, then $\ell(v_1)\ge \ell(v_2)$ if and only if $\ord_{q_1}\omega_{v_1}\ge -1$. Together with the previous property, this implies that $\ell(v_1)=\ell(v_2)$ if and only if $\ord_{q_1}\omega_{v_1}=-1$.

\smallskip

\item[(v)] If $q_1\in C_{v_1}$ and $q_2 \in C_{v_2}$, $v_1,v_2\in V(\Gamma_C)$, form a horizontal node $e\in E(\Gamma_C)$ (i.e. $\kappa_e=0$), then
\begin{equation}\label{eq:residue condition at horizontal nodes}
\res_{q_1} \omega_{v_1}+\res_{q_2} \omega_{v_2}=0.
\end{equation}

\smallskip

\item[(vi)] For every level $-1\leq j\leq -L$ of $\Gamma_C$ and for every connected component $Y$ of $C_{(>j)}$,
\begin{equation}\label{eq:residue condition at vertical nodes}
\sum_{q\in Y\cap C_{(j)}} \res_{q^-} \omega_{v(q^-)} =0,
\end{equation}
where $q^+\in Y$ and $q^-\in C_{(j)}$ form the vertical node $q \in Y\cap C_{(j)}$.
\end{itemize}

\medskip

\item a cyclic order reversing bijection $\sigma_q\colon P^\mathrm{in}_{q^-} \to P^\mathrm{out}_{q^+}$ for each vertical node $q$ formed by identifying $q^-$ on the upper level with $q^+$ on the lower level, where $\kappa_q=|P^\mathrm{in}_{q^-}|=|P^\mathrm{out}_{q^+}|$.
\end{enumerate}

\bigskip

\begin{remark}\label{remark:R-space}
Using notation from~\cite[Section~4.1]{CMZ20}, condition (2)(vi) is a reformulation of the $\mathfrak{R}$-global residue condition in the particular case when $\lambda$ is the partition of $H_p$ in one-element subsets and $\lambda_{\mathfrak{R}}=\lambda$.
\end{remark}

\bigskip

Lastly, there is an action of the universal cover of the torus $\mbC^{L}\to (\mbC^*)^{L}$ on multiscale residueless differentials by rescaling the differentials with strictly negative levels and rotating the prong matchings between levels accordingly, producing fractional Dehn twists. The stabilizer of this action is called the \emph{twist group} of the enhanced level graph and denoted by $\mathrm{Tw}_\Gamma$. Two multiscale residueless differentials  are defined to be equivalent if they differ by the action of $T_\Gamma:=\mbC^{L}/\mathrm{Tw}_\Gamma$. By further quotienting by the action of $\mbC^*$-rescaling the differentials on all levels and leaving all prong-matchings untouched, we obtain equivalence classes of projectivized multiscale residueless differentials.\\

As a special case of \cite[Proposition 4.2]{CMZ20} (corresponding to the choice of $\mathfrak{R}$ described in Remark~\ref{remark:R-space}), we have the following result.

\begin{proposition}\cite{CMZ20}\label{proposition:moduli of multiscale}
\begin{enumerate}[ 1.]

\item Given $a_1,\ldots,a_n \in \mbZ$, there is a proper smooth Deligne--Mumford stack $\oB^{\res}_g(a_1,\ldots,a_n)$ containing $B^{\res}_g(a_1,\ldots,a_n)$ as an open dense substack whose complement is a normal crossing divisor. $\oB^{\res}_g(a_1,\ldots,a_n)$ is a moduli stack for families of equivalence classes of projectivized multiscale residueless differentials. Its dimension is
$$
\dim \oB^{\res}_g(a_1,\ldots,a_n)= 2g-2+n-N_{a_{[n]}}.
$$

\medskip

\item We denote the closure of the stratum parameterizing multiscaled differentials whose enhanced level graph is $(\Gamma,\ell,\kappa)$ by $D_{(\Gamma,\ell,\kappa)}$ or simply by $D_\Gamma$. Then $D_\Gamma$ is a proper smooth closed substack of $\oB^{\res}_g(a_1,\ldots,a_n)$ of codimension
$$
\codim D_{\Gamma} = h+L,
$$
where~$h$ is the number of horizontal edges in $(\Gamma,\ell,\kappa)$ and~$L+1$ is the number of levels.
\end{enumerate}
\end{proposition}

\bigskip

There is a forgetful map $p\colon\oB^{\res}_g(a_1,\ldots,a_n)\to \oM_{g,n}$ associating to a projectivized multiscale differential on a stable curve $C$ the stable curve itself. It restricts to an isomorphism of Deligne--Mumford stacks $p\colon B^{\res}_g(a_1,\ldots,a_n)\to \cH_g^{\res}(a_1,\ldots,a_n)\subset\cM_{g,n}$ and, clearly, 
$$
[\oH_g^{\res} (a_1,\ldots,a_n)]= p_*[\oB^{\res}_g(a_1,\ldots,a_n)].
$$
We will use the above description of the boundary stratification of $\oB^{\res}_g(a_1,\ldots,a_n)$ to understand the intersection of $[\oH_g^{\res} (a_1,\ldots,a_n)]$ with the boundary stratum of stable curves with one separating node.\\

\subsection{The class $[\oH_g^{\res} (a_1,\ldots,a_n)]$ as a partial cohomological field theory}

Recall the following generalization from \cite{LRZ15} of the notion of cohomological field theory (CohFT) from~\cite{KM94}.

\begin{definition}
A \emph{partial CohFT} is a system of linear maps $c_{g,n}\colon V^{\otimes n} \to H^\even(\oM_{g,n})$, for all pairs of nonnegative integers $(g,n)$ in the stable range $2g-2+n>0$, where $V$ is an arbitrary finite dimensional $\mbC$-vector space, called the \emph{phase space}, together with a special element $e_\un\in V$, called the \emph{unit}, and a symmetric nondegenerate bilinear form $\eta\in (V^*)^{\otimes 2}$, called the \emph{metric}, such that, chosen any basis $\{e_\alpha\}_{\alpha\in A}$ of $V$, $|A|=\dim V$, the following axioms are satisfied:
\begin{itemize}
\item[(i)] The maps $c_{g,n}$ are equivariant with respect to the $S_n$-action permuting the $n$ copies of~$V$ in $V^{\otimes n}$ and the $n$ marked points in $\oM_{g,n}$, respectively.

\medskip

\item[(ii)] $\pi^* c_{g,n}( \otimes_{i=1}^n e_{\alpha_i}) = c_{g,n+1}(\otimes_{i=1}^n  e_{\alpha_i}\otimes e_\un)$ for $\alpha_1,\ldots,\alpha_n\in A$, where $\pi\colon\oM_{g,n+1}\to\oM_{g,n}$ is the map that forgets the last marked point.\\
Moreover $c_{0,3}(e_{\alpha}\otimes e_\beta \otimes e_\un) =\eta(e_\alpha\otimes e_\beta) =:\eta_{\alpha\beta}$ for $\alpha,\beta\in A$, where we identify $H^*(\oM_{0,3})=H^*(\pt)=\mbC$.

\medskip

\item[(iii)] $\gl^* c_{g_1+g_2,n_1+n_2}( \otimes_{i=1}^n e_{\alpha_i}) = c_{g_1,n_1+1}(\otimes_{i\in I} e_{\alpha_i} \otimes e_\mu)\eta^{\mu \nu} c_{g_2,n_2+1}( \otimes_{j\in J} e_{\alpha_j}\otimes e_\nu)$ for $2g_1-1+n_1>0$, $2g_2-1+n_2>0$, and $\alpha_1,\ldots,\alpha_n\in A$, where $I \sqcup J=[n]$, $|I|=n_1$, $|J|=n_2$, and $\gl\colon\oM_{g_1,n_1+1}\times\oM_{g_2,n_2+1}\to \oM_{g_1+g_2,n_1+n_2}$ is the corresponding gluing map and where~$\eta^{\alpha\beta}$ is defined by $\eta^{\alpha \mu}\eta_{\mu \beta} = \delta^\alpha_\beta$ for $\alpha,\beta\in A$.
\end{itemize}
\end{definition}

\bigskip

\begin{definition}
A \emph{CohFT} is a partial CohFT $c_{g,n}\colon V^{\otimes n} \to H^\even(\oM_{g,n})$ such that the following extra axiom is satisfied:
\begin{itemize}
\item[(iv)] $\gl^* c_{g+1,n}(\otimes_{i=1}^n e_{\alpha_i}) = c_{g,n+2}(\otimes_{i=1}^n e_{\alpha_i}\otimes e_{\mu}\otimes e_\nu) \eta^{\mu \nu}$, where  $\gl\colon\oM_{g,n+2}\to \oM_{g+1,n}$ is the gluing map, which increases the genus by identifying the last two marked points.
\end{itemize}
\end{definition}

\bigskip

\begin{definition}\label{definition:homogeneous partial CohFT}
A partial CohFT $c_{g,n}\colon V^{\otimes n} \to H^\even(\oM_{g,n})$ is called \emph{homogeneous} if~$V$ is a graded vector space with a homogeneous basis $\{e_\alpha\}_{\alpha\in A}$, with $q_\alpha:=\deg e_\alpha$, the metric $\eta$ on~$V$, seen as the map $\eta\colon V^{\otimes 2}\to \mbC$, is homogeneous with $\delta:=-\deg \eta$, $\deg e_\un=0$ and complex constants~$r^\alpha$ for $\alpha\in A$ and $\gamma$ exist such that the following condition is satisfied:
\begin{gather}\label{eq:homogeneous partial CohFT}
\Deg c_{g,n}(\otimes_{i=1}^ne_{\alpha_i})+\pi_* c_{g,n+1}(\otimes_{i=1}^ne_{\alpha_i}\otimes r^\alpha e_\alpha)=\left(\sum_{i=1}^n q_{\alpha_i}+\gamma g-\delta\right)c_{g,n}(\otimes_{i=1}^ne_{\alpha_i}),
\end{gather}
where $\Deg\colon H^*(\oM_{g,n})\to H^*(\oM_{g,n})$ is the operator that acts on $H^i(\oM_{g,n})$ by multiplication by~$\frac{i}{2}$ and $\pi\colon\oM_{g,n+1}\to\oM_{g,n}$ forgets the last marked point. The constant $\gamma$ 
is called the \emph{conformal dimension} of our partial CohFT.
\end{definition}

\bigskip

When a homogeneous partial CohFT is a CohFT, the loop axiom enforces the condition $\gamma=\delta$.\\

As remarked in \cite[Section 3]{BR21b}, a sufficient condition for the definition of a partial CohFT to make sense when $V$ is countably generated, say $V:=\mathrm{span}\left(\{e_\alpha\}_{\alpha\in\mbZ}\right)$, i.e. $A=\mbZ$ in the above definition, is that the set $\{\alpha_n\in \mbZ| c_{g,n}(\otimes_{i=1}^{n} e_{\alpha_i})\neq 0\}$ is finite for every $g,n$ in the stable range and $\alpha_1,\ldots,\alpha_{n-1}\in\mbZ$, and that $\eta_{\alpha \beta}$ has a unique two-sided inverse $\eta^{\alpha \beta}$.

\bigskip

Let us introduce the notation $\mbZ^\star:=\mbZ\setminus\{-1\}$.

\begin{proposition}\label{proposition:partial CohFT Hres}
Let $V:=\mathrm{span}\left(\{e_\alpha\}_{\alpha\in\mbZ^\star}\right)$ and let $\eta$ be the nondegenerate symmetric bilinear form on $V$ given by $\eta_{\alpha\beta}=\eta(e_\alpha\otimes e_\beta):=\delta_{\alpha+\beta,-2}$. For $g,n\geq 0$ and $2g-2+n>0$, the classes $c_{g,n}\colon V^{\otimes n} \to H^\even(\oM_{g,n})$, with
\begin{equation}\label{eq:partial CohFT Hres}
c_{g,n}(e_{\alpha_1}\otimes\ldots\otimes e_{\alpha_n}):=[\oH_g^{\res}(\alpha_1,\ldots,\alpha_n)] \in H^{2(g-1+N_{\alpha_{[n]}})}(\oM_{g,n}),\qquad \alpha_1,\ldots,\alpha_n\in\mbZ^\star,
\end{equation}
form an infinite rank homogeneous partial CohFT with unit $e_0$, metric $\eta$, and, with notations as in Definition~\ref{definition:homogeneous partial CohFT}, $q_\alpha=0$ if $\alpha\ge 0$ and $q_\alpha=1$ if $\alpha\le -2$, $r^\alpha=0$ for all $\alpha \in \mbZ^\star$, and $\gamma=\delta=1$.
\end{proposition}
\begin{proof}
First note that for fixed $g$, $n$ in the stable range and $\alpha_1, \dots, \alpha_{n-1} \in \mathbb{Z}^*$ the set $\{\alpha_n\in\mbZ^\star| c_{g,n}(\otimes_{i=1}^{n} e_{\alpha_i})\neq 0\}$ is indeed finite (actually composed of one element) thanks to the fact that $[\oH_g^{\res}(\alpha_1,\ldots,\alpha_n)]=0$ unless $\sum_{i=1}^n \alpha_i = 2g-2$. Further, $\eta_{\alpha\beta}=\delta_{\alpha+\beta,-2}$ has a unique two-sided inverse, namely $\eta^{\alpha\beta}=\delta_{\alpha+\beta,-2}$.\\

$S_n$-equivariance of the linear maps $c_{g,n}$ is clear from the definition.\\

On the marked curve $(\mbC\mbP^1;0,\infty,1)$ a (unique up to a multiplicative constant) meromorphic differential whose divisor is $\alpha[0]+\beta[\infty]+0[1]$ exists if $\beta=-\alpha-2$ and is given by $\omega=z^\alpha dz$, which shows that $c_{0,3}(e_\alpha\otimes e_\beta \otimes e_0) = \delta_{\alpha+\beta,-2}$. Let us compute $c_{g,n+1}(\otimes_{i=1}^n e_{\alpha_i}\otimes e_0)$ when $2g-2+n>0$. Consider the lift $\widetilde{\pi}\colon\oB^{\res}_g(\alpha_1,\ldots,\alpha_n,0)\to\oB^{\res}_g(\alpha_1,\ldots,\alpha_n)$ of $\pi\colon\oM_{g,n+1}\to\oM_{g,n}$ through $p\colon\oB^{\res}_g(\alpha_1,\ldots,\alpha_n)\to \oM_{g,n}$. Since $\oB_g^{\res}(\alpha_1,\ldots,\alpha_n,0)$ is the moduli stack of projectivized multiscale differentials where the last marked point is unconstrained (neither a zero nor a pole), we have that $\widetilde{\pi}$ is faithfully flat. Consider then the fiber product $X$ of $\oB^{\res}_g(\alpha_1,\ldots,\alpha_n)$ and $\oM_{g,n+1}$ over $\oM_{g,n}$, denoting the two projections by $a$ and $b$, respectively. Since $\pi$ is faithfully flat and $p$ is proper, then $a$ is faithfully flat and $b$ is proper and we have $\pi^*p_*=b_*a^*$ in the Chow group. Moreover the maps $\widetilde{\pi}$ and $p$ induce a proper birational morphism $f\colon\oB^{\res}_g(\alpha_1,\ldots,\alpha_n,0)\to X$ with $p =b f$ and $\widetilde{\pi}= af$. Now, always working in the Chow group, we have $\widetilde{\pi}^*[\oB^{\res}_g(\alpha_1,\ldots,\alpha_n)] = [\oB_g^{\res}(\alpha_1,\ldots,\alpha_n,0)]$ and $a^*[\oB^{\res}_g(\alpha_1,\ldots,\alpha_n)]=[X]$ by faithful flatness of $\widetilde{\pi}$ and $a$, while $f_*[\oB_g^{\res}(\alpha_1,\ldots,\alpha_n,0)] = [X]$ by birationality of $f$. Then we conclude that $c_{g,n+1}(\otimes_{i=1}^n e_{\alpha_i}\otimes e_0)=p_*[\oB_g^{\res}(\alpha_1,\ldots,\alpha_n,0)]=p_*\widetilde{\pi}^*[\oB^{\res}_g(\alpha_1,\ldots,\alpha_n)] = b_*f_* \widetilde{\pi}^* [\oB^{\res}_g(\alpha_1,\ldots,\alpha_n)]= b_*a^*[\oB^{\res}_g(\alpha_1,\ldots,\alpha_n)]=\pi^*p_*[\oB^{\res}_g(\alpha_1,\ldots,\alpha_n)]=\pi^*c_{g,n}(\otimes_{i=1}^n e_{\alpha_i})$ in Chow and hence in cohomology.\\

Next, we are interested in $\sigma^*c_{g,n}(\otimes_{i=1}^n e_{\alpha_i})$ where $\sigma\colon\oM_{g_1,|I|+1}\times\oM_{g_2,|J|+1}\to\oM_{g,n}$ is the natural boundary map with $g_1+g_2=g$ and $I\sqcup J=[n]$. The preimage $p^{-1}\left( \sigma\left(\oM_{g_1,|I|+1}\times\oM_{g_2,|J|+1}\right)\right)$ is a normal crossing divisor of $\oB^{\res}_g(\alpha_1,\ldots,\alpha_n)$,  which is the union of strata  of the form~$D_{\Gamma}$ with $\Gamma$ being either a one level connected graph with two vertices and one horizontal edge, a two level connected graph with one vertex per level, one vertical edge and no horizontal edges, or a two level connected graph with at least two vertices on at least one of the levels and no horizontal edges.\\

In the first case $D_{\Gamma}$ is actually empty: horizontal nodes correspond to simple poles and these are forbidden by the residue theorem, since all other poles are at marked points, where residues are set to zero.\\

In the third case the  stratum $D_\Gamma$ projects to a stratum of $\oH^{\res}_g(\alpha_1,\ldots,\alpha_n)$ of codimension at least $2$ because the fibers of $p|_{D_\Gamma}$ are of dimension at least $1$  (given a multiscale differential whose underlying level graph has at least two vertices on the same level not connected by horizontal nodes, one can always rescale the meromorphic differential on one vertex relative to the ones on vertices of the same level without changing the underlying stable curve).\\

 In the second case, notice that if $D_\Gamma\ne\emptyset$, then for the only edge $e\in E(\Gamma)$ identifying the two points $q^-\in C_{(-1)}$ and $q^+\in C_{(0)}$, we have $\kappa_e=|2g_1-1-\sum_{i\in I}a_i|=|2g_2-1-\sum_{j\in J}a_j|\ne 0$ and $\res_{q^-}\omega_{v(q^-)}=0$, and moreover $2g_1-1-\sum_{i\in I}a_i$ is positive if and only if the vertex of~$\Gamma$ of level~$0$ is incident to the legs marked by $I$. Since $T_{\Gamma}=\mbC^*$ in this case, this shows that there is a morphism $\widetilde{\sigma}\colon\oB^{\res}_{g_1}(\alpha_I,2g_1-2-\sum_{i\in I}a_i)\times\oB^{\res}_{g_2}(\alpha_J,2g_2-2-\sum_{j\in J}a_j)\to \oB^{\res}_g(\alpha_1,\ldots,\alpha_n)$ lifting $\sigma$, which is an isomorphism onto its image $D_{\Gamma}$, and therefore $\sigma^{-1}\left(\oH^{\res}_g(\alpha_1,\ldots,\alpha_n)\right)=\oH^{\res}_{g_1}(\alpha_I,2g_1-2-\sum_{i\in I}a_i)\times\oH^{\res}_{g_2}(\alpha_J,2g_2-2-\sum_{j\in J}a_j)$.\\

The above considerations show that, denoting $\kappa:=2g_1-1-\sum_{i\in I}a_i$, we have
$$
\sigma^*c_{g,n}(\otimes_{i=1}^n e_{\alpha_i})=
\begin{cases}
0,&\text{if $\kappa=0$},\\
m\, c_{g_1,|I|+1}(\otimes_{i\in I}e_{\alpha_i}\otimes e_{\kappa-1})c_{g_2,|J|+1}(\otimes_{j\in J}e_{\alpha_j}\otimes e_{-\kappa-1}),&\text{if $\kappa\ne 0$},
\end{cases}
$$
and the fact that $m=1$ in the second case is equivalent to the fact that the intersection of~$\oH^{\res}_g(\alpha_1,\ldots,\alpha_n)$ with the image of $\sigma$ along $\oH_{g_1}^{\res}(\alpha_I,\kappa-1)\times \oH^{\res}_{g_2}(\alpha_J,-\kappa-1)$ is generically transversal.\\

Denote by $S_1$ and $S_2$ the smooth parts of $\oH^{\res}_{g_1}(\alpha_I,\kappa-1)$ and $\oH^{\res}_{g_2}(\alpha_J,-\kappa-1)$, respectively. Denote also $S:=\oH_g^{\res}(\alpha_1,\ldots,\alpha_n)$ for brevity. Let us show that the intersection of $S$ with the image of $\sigma$ is transversal along $S_1 \times S_2$.\\

Pick points $p_1 \in S_1$ and $p_2 \in S_2$. Denote by $p \in \oM_{g,n}$ the point $\sigma(p_1, p_2)$. By the smoothness of stratum $S_1$, we can choose local coordinates $U_1 \times V_1$ on $\oM_{g_1,|I|+1}$ in the neighborhood of~$p_1$ such that $S_1 = U_1 \times \{0\}$. We choose local coordinates $U_2 \times V_2$ in the neighborhood of $p_2$ in~$\oM_{g_2, |J|+1}$ in the same way. Denote by $\Delta \subset \mathbb{C}$ the unit disc. We claim that we can choose local coordinates $U_1 \times V_1 \times U_2 \times V_2 \times \Delta$ on $\oM_{g,n}$ in the neighborhood of $p$ such that the stratum~$S$ is $U_1 \times \{0\} \times U_2 \times \{0\} \times \Delta$ and the image of $\sigma$ is $U_1 \times V_1 \times U_2 \times V_2 \times \{ 0 \}$. The transversality of the intersection is then obvious. So let us describe the choice of local coordinates.\\

Every curve $C_1$ in $U_1 \times \{0\}$ carries a residueless meromorphic differential. It is unique up to a multiplicative constant. Choose this constant in some way over $U_1$ and denote the meromorphic differential by~$\alpha$. Similarly, denote by $\beta$ the meromorphic differential on a curve $C_2$ of $U_2 \times \{0\}$. At the marked points to be glued into a node there is a local coordinate $z$ on $C_1$ and $w$ on $C_2$ such that $\alpha = d(z^k)$, $\beta = d (w^{-k})$. The choice of such local coordinates is unique up to the multiplication by a $k$th root of unity; we fix one uniform choice over all of $U_1$ and~$U_2$. We extend the local coordinates $z$ and $w$ to curves in $U_1 \times V_1$ and $U_2 \times V_2$ in an arbitrary way. Now, to a curve $C_1 \in U_1 \times V_1$, a curve $C_2 \in U_2 \times V_2$, and a number $\varepsilon \in \Delta$ we assign the curve obtained by removing the neighborhoods of the marked points $z=0$ and $w=0$ and gluing in the ``waist'' $zw = \varepsilon$. In the case when $C_1 \in U_1 \times \{0\}$ and $C_2 \in U_2 \times \{0\}$, the curve thus obtained does carry a residueless meromorphic differential, because $\alpha$ and $\varepsilon^k \beta$ agree on the waist. Thus the stratum~$S$ is indeed given by $U_1 \times \{0\} \times U_2 \times \{0\} \times \Delta$, while the image of $\sigma$ is $\{ \varepsilon =0 \}$.\\

 We conclude that $\sigma^*c_{g,n}(\otimes_{i=1}^n e_{\alpha_i}) = \sum_{\alpha\in \mbZ^\star}c_{g_1,|I|+1}(\otimes_{i\in I}e_{\alpha_i}\otimes e_\alpha)c_{g_2,|J|+1}(\otimes_{j\in J}e_{\alpha_j}\otimes e_{-\alpha-2})$, as required.\\

Finally, from formula~\eqref{eq:dim Hres} we obtain $\Deg c_{g,n}(\otimes_{i=1}^n e_{\alpha_i})= (g-1+N_{\alpha_{[n]}})c_{g,n}(\otimes_{i=1}^n e_{\alpha_i})$, which shows that with the constants $q_\alpha=0$ if $\alpha\geq 0$ and $q_\alpha=1$ if $\alpha\leq -2$, and $\gamma=\delta=1$, which are compatible with $\deg e_0=0$ and $\deg \eta=-\delta$, equation~\eqref{eq:homogeneous partial CohFT} is satisfied, thus completing the proof.
\end{proof}

\bigskip

\section{The DR hierarchy for the cycle of residueless meromorphic differentials}\label{section:DR hierarchy}

Here we briefly review the notion of double ramification (DR) hierarchy for a partial CohFT and then apply this construction to the partial CohFT formed by the cycles of residueless meromorphic differentials.\\

In~\cite{Bur15}, the first named author introduced a construction associating an integrable Hamiltonian system of evolutionary PDEs to a given CohFT. In \cite{BDGR18} it was proved that the same construction also works for partial CohFTs and, in \cite{BR21b}, the first example of DR hierarchy associated to an infinite rank partial CohFT was computed. Finally, in \cite{BR21a,ABLR21}, the construction was generalized to associate an integrable system of evolutionary PDEs to any F-CohFT (a generalization of the notion of partial CohFT introduced in \cite{BR21a} and further studied in \cite{ABLR20}). Although this last generalization will not be needed in this paper, it has several points in common with a reduction of the DR hierarchy associated to the infinte rank partial CohFT \eqref{eq:partial CohFT Hres} (the reduction corresponding to only considering the spaces of meromorphic differentials with exactly two zeros), which we will study in Section~\ref{section:two zeros}.\\

Let~$\psi_i\in H^2(\oM_{g,n})$ be the $i$-th \emph{psi class}, i.e. the first Chern class of the tautological line bundle over~$\oM_{g,n}$ whose fiber at a stable curve is the cotangent line at its $i$-th marked point. Let~$\lambda_j\in H^{2j}(\oM_{g,n})$ be the $j$-th \emph{Hodge class}, i.e. the $j$-th Chern class of the Hodge bundle~$\mathbb E$, which is the rank~$g$ vector bundle over~$\oM_{g,n}$ whose fiber at a stable curve is its space of holomorphic one-forms.\\

For any $a_1,\dots,a_n\in \mbZ$, $\sum_{i=1}^n a_i =0$, let $\DR_g(a_1,\ldots,a_n) \in H^{2g}(\oM_{g,n})$ be the (untwisted) {\it double ramification (DR) cycle}. The DR cycle is the pushforward, through the forgetful map to $\oM_{g,n}$, of the virtual fundamental class of the moduli space of projectivized stable maps to~$\CP^1$ relative to~$0$ and~$\infty$, with ramification profile $a_1,\ldots,a_n$ at the marked points (see, e.g.,~\cite{BSSZ15} for more details). 
More precisely, the pushforward itself lies in $H_{2(2g-3+n)}(\oM_{g,n})$, while its Poincar\'e dual cohomology class lies in $H^{2g}(\oM_{g,n})$. By abuse of notation, we will denote both the pushforward and its Poincar\'e dual by $\DR_g(a_1,\ldots,a_n)$.\\ 

The restriction $\DR_g(a_1,\ldots,a_n)\big|_{\cM_{g,n}^{\ct}}$, where we recall that $\cM_{g,n}^{\ct}$ is the moduli space of stable curves of compact type, is a homogeneous polynomial in $a_1,\ldots,a_n$ of degree $2g$ with the coefficients in~$H^{2g}(\cM_{g,n}^{\ct})$ (see, e.g.,~\cite{JPPZ17}). Polynomiality of the DR cycle on~$\cM^\ct_{g,n}$ together with the fact that~$\lambda_g$ vanishes on~$\oM_{g,n}\setminus\cM_{g,n}^{\ct}$ (see, e.g.,~\cite[Section~0.4]{FP00}) implies that the cohomology class $\lambda_g\DR_g\left(-\sum_{j=1}^n a_j,a_1,\ldots,a_n\right) \in H^{4g}(\oM_{g,n+1})$ is a degree $2g$ homogeneous polynomial in the coefficients $a_1,\ldots,a_n$.\\

Let $\hcA_A$ and $\hLambda_A$ be the spaces of \emph{differential polynomials} and \emph{local functionals} in formal variables~$u^\alpha_k$, $\alpha \in A$, $k\geq 0$, and $\eps$, where $A$ is an index set (as above, finite or countable) and, in the case of finite~$A$, the definitions and the notations are taken from the paper~\cite[Section 2.1]{Ros17}. A minor adjustment is needed in order to include the case of countable $A$ in our considerations. The ring $\mbC[[u^*_*]]$ is graded by the differential grading $\deg_{\d_x} u^\alpha_k:=k$, and the degree $d$ part of it is denoted by $\cA_A^{[d]}$. We then define $\cA_A:=\oplus_{d\ge 0}\cA_A^{[d]}$, $\hcA_A:=\cA_A[[\eps]]$, and $\hLambda_A:=\hcA_A\left/\left(\d_x \hcA_A \oplus \mbC[[\eps]]\right)\right.$ where $\d_x:=\sum_{k\geq 0}u^\alpha_{k+1} \frac{\d}{\d u^\alpha_k}$. We denote the image of $f\in\hcA_A$ through the natural projection to $\hLambda_A$ by $\of=\int f dx$. Assigning $\deg_{\d_x}\eps:=-1$, the degree~$d$ parts of~$\hcA_A$ and~$\hLambda_A$ are denoted by~$\hcA_A^{[d]}$ and~$\hLambda^{[d]}_A$, respectively.\\

\begin{remark}
In the case of finite $A$ we have $\hcA_A=\mbC[[u^*]][u^*_{>0}][[\eps]]$ where $u^\alpha:=u^\alpha_0$, which is the standard way to introduce the space of differential polynomials, but for countable $A$ we have $\hcA_A\ne\mbC[[u^*]][u^*_{>0}][[\eps]]$. 
\end{remark}

\bigskip

Given a partial CohFT $c_{g,n}\colon V^{\otimes n} \to H^\even(\oM_{g,n})$ with $V=\spn\left(\{e_\alpha\}_{\alpha \in A}\right)$, unit $e_\un$, and metric $\eta$, the Hamiltonian densities for the associated DR hierarchy are the generating series~\cite{BR16a}
\begin{equation}\label{eq:DR densities}
g_{\alpha,d}:=\sum_{\substack{g,n\geq 0\\
2g-1+n>0}}\frac{\eps^{2g}}{n!}\sum_{k_1,\ldots,k_n \geq 0}\Coef_{a_1^{k_1}\ldots a_n^{k_n}}\left(\int_{\DR_g\left(-\sum_{i=1}^n a_i,a_1\ldots,a_n\right)}\hspace{-2.cm}\lambda_g \psi_1^d c_{g,n+1}(e_\alpha\otimes\otimes_{i=1}^n e_{\alpha_i})\right)\prod_{i=1}^n u^{\alpha_i}_{k_i} \in \hcA^{[0]}_A,
\end{equation}
where $\alpha\in A$ and $d\in \mbZ_{\geq 0}$. To this definition one can add $g_{\alpha,-1}:=\eta_{\alpha\mu}u^\mu$, $\alpha\in A$. The \emph{Hamiltonians} of the DR hierarchy are the local functionals $\og_{\alpha,d}\in \hLambda^{[0]}_A$, $\alpha\in A$, $d\geq -1$. By a result of \cite{Bur15}, the Hamiltonians of the DR hierarchy are in involution with respect to the Poisson brackets on $\hLambda_A$ defined by $\{\of,\og\}=\int \left(\frac{\delta \of}{\delta u^\alpha} \eta^{\alpha\beta} \d_x \frac{\delta \og}{\delta u^\beta}\right) dx$ for any two local functionals $\of,\og \in \hLambda_A$, that is $\{\og_{\alpha_1,d_1},\og_{\alpha_2,d_2}\} =0$ for all $\alpha_1,\alpha_2 \in A$ and $d_1,d_2\geq -1$.\\

This implies that the infinite system of evolutionary PDEs, called the \emph{DR hierarchy},
\begin{equation}\label{eq:DR equations}
\frac{\d u^\alpha}{\d t^\beta_d} = \eta^{\alpha\mu} \d_x \frac{\delta \og_{\beta,d}}{\delta u^\mu}, \qquad \alpha,\beta \in A,\quad d\geq 0,
\end{equation}
where, for any $\of\in \hLambda_A$,
$$
\frac{\delta \of}{\delta u^\alpha}:=\sum_{k\geq 0} (-\d_x)^k\frac{\d f}{\d u^\alpha_k}, \qquad \alpha\in A,
$$
satisfies the compatibility conditions $\frac{\d}{\d t^{\beta_2}_{d_2}}\frac{\d u^\alpha}{\d t^{\beta_1}_{d_1}}=\frac{\d}{\d t^{\beta_1}_{d_1}}\frac{\d u^\alpha}{\d t^{\beta_2}_{d_2}}$ for all $\alpha,\beta_1,\beta_2 \in A$ and $d_1,d_2\geq 0$.\\

In \cite{BR16a, BDGR18, BDGR19} the authors showed that the DR hierarchy of a partial CohFT is a hierarchy of DR type, which means in particular that it is a tau-symmetric Hamiltonian system and its Hamiltonian densities can be reconstructed uniquely from the Hamiltonian $\og_{\un,1}$ only, via a universal recursion equation.\\

If the partial CohFT $c_{g,n}\colon V^{\otimes n} \to H^\even(\oM_{g,n})$ is homogeneous, with notations as in Definition \ref{definition:homogeneous partial CohFT}, consider the Euler differential operator on $\hcA_A$
\begin{equation*}
\hE:=\sum_{k\geq 0}\left(\left(1-q_\alpha\right)u^\alpha_k+\delta_{k,0} r^\alpha\right) \frac{\d}{\d u^\alpha_k}+\frac{1-\gamma}{2} \eps \frac{\d}{\d \eps}.
\end{equation*}
Then it follows easily from dimension counting in the integral appearing in equation \eqref{eq:DR densities} that
\begin{align}\label{eq:Euler on DR densities}
\hE (g_{\alpha,d}) = (d+2+q_\alpha-\delta) g_{\alpha,d}+r^\mu c_\mu^{\alpha\nu}g_{\nu,d-1}, \qquad \alpha\in A,\quad d\geq 0,
\end{align}
where $c_\mu^{\alpha\nu}:=\eta^{\alpha \beta} \eta^{\nu\gamma} c_{0,3}(e_\mu\otimes e_\beta\otimes e_\gamma)\in \mbC$ for all $\mu,\alpha,\nu \in A$.\\

Let us apply the DR hierarchy construction to the partial CohFT of Proposition \ref{proposition:partial CohFT Hres}.

\begin{proposition}
Let us endow the ring $\hcA_{\mbZ^\star}$ with the triple grading
\begin{equation}\label{eq:degree of variables}
\odeg u^{\alpha}_k:= 
\begin{cases}
(k,1,-\alpha),&\text{if $\alpha\geq 0$},\\ 
(k,0,-\alpha),&\text{if $\alpha\leq -2$},
\end{cases} \qquad \odeg\eps: = (-1,0,1). 
\end{equation}
Then the Hamiltonian densities of the DR hierarchy associated to the homogeneous partial CohFT of Proposition \ref{proposition:partial CohFT Hres} satisfy
\begin{equation}\label{eq:degree of DR densities}
\odeg g_{\alpha,d}=
\begin{cases}
(0,d+1,\alpha+2),&\text{if $\alpha\geq 0$},\\
(0,d+2,\alpha+2),&\text{if $\alpha\le -2$},
\end{cases} \qquad d\geq -1.
\end{equation}
\end{proposition}
\begin{proof} 
Notice that the first entry in the triple degree $\odeg$ coincides with $\deg_{\d_x}$. The three entries in the triple degree of equation \ref{eq:degree of DR densities} then follow easily from the fact that $g_{\alpha,d}\in \hcA^{[0]}_{\mbZ^\star}$, from equation \eqref{eq:Euler on DR densities}, and from the fact that $c_{g,n}(e_\alpha\otimes\otimes_{i=1}^n e_{\alpha_i})=0$ unless $-\sum_{i=1}^n \alpha_i +2g = \alpha+2$, respectively.
\end{proof}


\bigskip

\section{A reduction to meromorphic differentials with two zeros and the KP hierarchy}\label{section:two zeros}\label{section:KP}

In this section we describe a reduction of the DR hierarchy for the cycles of residueless meromorphic differentials. As we will see, this reduction does not respect the Poisson structure, in the sense that it is only defined at the level of vector fields. As the main result of the paper, we will prove that the reduction coincides with the KP hierarchy up to a Miura transformation.\\ 

\subsection{A reduction of the DR hierarchy}

Consider the DR hierarchy for the partial CohFT formed by the cycles of residueless meromorphic differentials with:
\begin{gather}\label{eq:DR for meromorphic}
\frac{\d u^{\alpha}}{\d t^\beta_d} = \d_x \frac{\delta\og_{\beta,d}}{\delta u^{-\alpha-2}}, \qquad \alpha,\beta\in\mbZ^\star,\quad d\geq 0.
\end{gather}

\begin{proposition}
The subset of flows of the DR hierarhy \eqref{eq:DR for meromorphic}
\begin{equation}\label{eq:primary flows}
\frac{\d u^{\alpha}}{\d t^\beta_0} = \d_x \frac{\delta \og_{\beta,0}}{\delta u^{-\alpha-2}}, \qquad \alpha\in \mbZ^\star,\quad \beta\geq 0,
\end{equation}
preserves the submanifold $\{u^\alpha_k=0, \alpha,k\geq 0\}$.
\end{proposition}
\begin{proof}
The statement is equivalent to 
$$
\left. \frac{\d u^{\alpha}}{\d t^\beta_0}  \right|_{u^{\geq 0}_*=0}= \left. \d_x \frac{\delta \og_{\beta,0}}{\delta u^{-\alpha-2}}\right|_{u^{\geq 0}_*=0}=0, \qquad \alpha,\beta\geq 0.
$$
Since, by \eqref{eq:degree of DR densities}, $\odeg g_{\beta,0}=(0,1,\beta+2)$ for $\beta\geq 0$ and, by \eqref{eq:degree of variables}, $\odeg u^{-\alpha-2}_k=(k,0,\alpha+2)$ for $\alpha\geq 0$, we have
$$
\odeg\frac{\d g_{\beta,0}}{\d u^{-\alpha-2}_k}=(-k,1,\beta-\alpha), \qquad \alpha,\beta\geq 0.
$$
But, again, $\odeg u^\gamma_k=(k,0,-\gamma)$ for $\gamma\le -2$, which implies 
$$
\left. \frac{\d g_{\beta,0}}{\d u^{-\alpha-2}_k}\right|_{u^{\geq 0}_*=0}=0, \qquad \alpha,\beta\geq 0.
$$
This implies
$$
\left. \frac{\delta \og_{\beta,0}}{\delta u^{-\alpha-2}}\right|_{u^{\geq 0}_*=0}=0, \qquad \alpha,\beta\geq 0,
$$
as desired.
\end{proof}

\bigskip

Let us summarize our considerations regarding the above reduction and also introduce more convenient notation.\\

Let $u_\alpha^{(k)}:=u^{-\alpha-1}_k$, $u_\alpha:= u_\alpha^{(0)}$, and $t^\alpha:= t^{\alpha-1}_0$, for $\alpha\geq 1$, $k\geq 0$. Consider the ring $\cR_u:=\mbC[u_*^{(*)}]$ and the following three gradings on it:
\begin{itemize}
\item The differential grading $\deg_{\d_x} u^{(k)}_\alpha:=k$. The corresponding homogeneous component of~$\cR_u$ of degree $d$ will be denoted by $\cR^{[d]}_u$. 

\medskip

\item A grading $\deg$ is given by $\deg u^{(k)}_\alpha:=\alpha+1+k$.

\medskip

\item A grading $\tdeg$ is given by $\tdeg u^{(k)}_\alpha:=1$. The corresponding homogeneous component of~$\cR_u$ of degree $d$ will be denoted by $\cR_{u;d}$. We will also use the notation $\cR_{u;\ge l}:=\bigoplus_{d\ge l}\cR_{u;d}$.
\end{itemize}
Let $\cR_u^\ev:=\bigoplus_{d\ge 0}\cR^{[2d]}_u$. We extend the three gradings to the ring $\hcR_u:=\cR_u[\eps]$ by
$$
\deg_{\d_x}\eps:=-1,\qquad \deg\eps:=0,\qquad \tdeg\eps:=0.
$$
Let $\hcR^\ev_u:=\cR^\ev_u[\eps]$.\\

\begin{theorem}\label{theorem:reduction}
For two integers $\alpha, \beta \geq 1$, consider the generating series
\begin{align}
P_{\alpha\beta}:=&\sum_{g\geq 0,\,n\geq 1}\frac{\eps^{2g}}{n!}\sum_{k_1,\ldots,k_n\geq 0}\prod_{i=1}^n u_{\alpha_i}^{(k_i)}\times\label{eq:polynomials P}\\
&\times\Coef_{a_1^{k_1}\ldots a_n^{k_n}}\left(\int_{\DR_g\left(-\sum_{i=1}^n a_i,0,a_1\ldots,a_n\right)}\hspace{-2.3cm}\lambda_g\left[\oH^{\res}_g(\alpha-1,\beta-1,-\alpha_1-1,\ldots,-\alpha_n-1)\right]\right).\notag
\end{align}
Then $P_{\alpha\beta}\in\hcR_{u;\ge 1}^{\ev;[0]}$ with $\deg P_{\alpha\beta}=\alpha+\beta$ and the system of equations
\begin{equation}\label{eq:reduction}
\frac{\d u_\alpha}{\d t^\beta} = \d_x P_{\alpha\beta}, \qquad \alpha,\beta\geq 1,
\end{equation}
satisfies the compatibility condition $\frac{\d}{\d t^{\beta_2}}\frac{\d u_\alpha}{\d t^{\beta_1}}=\frac{\d}{\d t^{\beta_1}}\frac{\d u_\alpha}{\d t^{\beta_2}}$ for all $\alpha,\beta_1,\beta_2 \geq 1$. Moreover, the polynomials $P_{\alpha\beta}$ satisfy the property
\begin{gather}
P_{1,\beta}-u_\beta\in\Im\left(\d_x^2\right),\label{eq:property of P}
\end{gather}
\end{theorem}
\begin{proof}
The system \eqref{eq:reduction} is nothing but the restriction of the system \eqref{eq:primary flows} to the submanifold $\{u^\alpha_k=0, \alpha,k\geq 0\}$, expressed in the new variables $u_\alpha^{(k)}$, $\alpha\geq 1$, $k\geq 0$, which form a system of coordinates on it. Compatibility and degree conditions follow from those for the DR hierarchy via the change of coordinates. In particular the degree conditions guarantee that~$P_{\alpha\beta}$ belongs to the subring~$\hcR_{u;\ge 1}^{\ev;[0]}$ of the ring~$\mbC[[u_*^{(*)}]][[\eps]]$, for all $\alpha,\beta\geq 1$.\\

Equation~\eqref{eq:property of P} follows from \eqref{eq:polynomials P} where, for $\alpha=1$ and unless $g=0$ and $n=1$, we have 
\begin{align*}
&\int_{\DR_g\left(-\sum_{i=1}^n a_i,0,a_1\ldots,a_n\right)}\hspace{-2.3cm}\lambda_g\left[\oH^{\res}_g(0,\beta-1,-\alpha_1-1,\ldots,-\alpha_n-1)\right]=\\
=&\int_{\pi_*\DR_g\left(-\sum_{i=1}^n a_i,0,a_1\ldots,a_n\right)}\hspace{-2.3cm}\lambda_g\left[\oH^{\res}_g(\beta-1,-\alpha_1-1,\ldots,-\alpha_n-1)\right],
\end{align*}
where $\pi\colon\oM_{g,n+2}\to\oM_{g,n+1}$ forgets the first marked point, and from the fact, proven in \cite[Lemma 5.1]{BDGR18}, that $\lambda_g\pi_*\DR_g\left(-\sum_{i=1}^n a_i,0,a_1\ldots,a_n\right)$ is a polynomial in the variables $a_1,\ldots,a_n$ divisible by $\left(\sum_{i=1}^n a_i\right)^2$.
\end{proof}

\bigskip

\subsection{The Miura transformation} 

The degree condition $\deg P_{1,\alpha}=\alpha+1$ together with the property~\eqref{eq:property of P} implies that the difference $P_{1,\alpha}-u_\alpha$ depends only on the variables $u^{(*)}_\beta$ with $\beta\le \alpha-2$ and on $\eps$. Therefore, the polynomial change of variables $u_\alpha\mapsto v_\alpha\big(u^{(*)}_*,\eps\big):=P_{1,\alpha}$ is invertible. We refer to this change of variables as Miura transformation, following the terminology of \cite{DZ01}.\\ 

Since $P_{1,\alpha}-u_\alpha\in\Im(\d_x)$, the system~\eqref{eq:reduction} has the following form in the new variables $v_\alpha$, $\alpha\ge 1$:
\begin{gather}\label{eq:v-system}
\frac{\d v_\alpha}{\d t^\beta}=\d_x Q_{\alpha\beta},
\end{gather}
where, by the theorem,
\begin{align}
&Q_{\alpha\beta}\in\hcR_{v;\ge 1}^{\ev;[0]},\label{eq:property1}\\
&\deg Q_{\alpha\beta}=\alpha+\beta,\label{eq:property2}\\
&Q_{\alpha,1}=Q_{1,\alpha}=v_\alpha,\label{eq:property3}\\
&Q_{\alpha\beta}=Q_{\beta\alpha}.\label{eq:property4}
\end{align}

\bigskip

\subsection{The KP hierarchy}

Let us briefly recall the construction of the KP hierarchy and some of its properties. A more detailed introduction can be found, for example, in~\cite{Dic03}.\\ 

Consider formal variables~$f_i^{(j)}$, $i\ge 1$, $j\ge 0$, and the associated ring $\cR_f$. A \emph{pseudo-differential operator} $A$ is a Laurent series
$$
A=\sum_{n=-\infty}^m a_n\d_x^n,\qquad m\in\mbZ,\quad a_n\in\cR_f.
$$
Let $A_+:=\sum_{n=0}^m a_n\d_x^n$ and $\res A:=a_{-1}$. The product of pseudo-differential operators is defined by the following commutation rule:
\begin{gather*}
\d_x^k\circ a\coloneqq\sum_{l=0}^\infty\frac{k(k-1)\ldots(k-l+1)}{l!}(\d_x^l a)\d_x^{k-l},\qquad a\in\cR_f,\quad k\in\mbZ,
\end{gather*}
which endows the space of pseudo-differential operators with the structure of an associative algebra.\\

Let 
$$
L:=\d_x+\sum_{i\ge 1}f_i\d_x^{-i}.
$$
The \emph{KP hierarchy} is the system of evolutionary PDEs with dependent variables $f_i$ defined by
$$
\frac{\d L}{\d T_n}=[(L^n)_+,L],\qquad n\ge 1.
$$
\begin{example}\label{example:L2}
Using that
\begin{gather*}
L^2=\d_x^2+2f_1+\left(2f_2+f_1^{(1)}\right)\d_x^{-1}+\left(2f_3+f_1^2+f_2^{(1)}\right)\d_x^{-2}+\ldots,
\end{gather*}
we compute
\begin{align*}
&\frac{\d f_1}{\d T_2}=2f_2^{(1)}+f_1^{(2)},\\
&\frac{\d f_2}{\d T_2}=2f_3^{(1)}+2 f_1 f_1^{(1)}+f_2^{(2)}.
\end{align*}
\end{example}

\bigskip

We can extend the grading $\deg$ from the ring $\cR_f$ to the ring of pseudo-differential operators by assigning~$\deg\d_x:=1$. We then obtain $\deg L=1$ and therefore $\deg L^k=k$, $\deg [L^k_+,L]=k+1$, which implies that the equations of the KP hierarchy have the form 
$$
\frac{\d f_i}{\d T_k}=S_{i,k},\qquad S_{i,k}\in\cR_{f;\ge 1},
$$ 
where $\deg S_{i,k}=i+k+1$.

\bigskip

We also see that $\deg\res L^k=k+1$, for $k\ge 1$, and
$$
\frac{\d}{\d f_k}\res L^k=\sum_{a+b=k-1}\res(L^a\circ\d_x^{-k}\circ L^b)=k.
$$
Therefore, $\res L^k-k f_k$ depends only on the variables $f^{(l)}_a$ with $a\le k-1$, which implies that the polynomial change of variables $f_\alpha\mapsto w_\alpha(f_*^{(*)}):=\res L^\alpha$, $\alpha\ge 1$, is invertible. Note also that
$$
\int\frac{\d}{\d T_n}\res L^a dx=\int\res\left(\frac{\d}{\d T_n}L^a\right)dx=\int\res[(L^n)_+,L^a]dx=0,
$$
where the last equality follows from the fact that $\int\res[A,B]dx=0$ for any two pseudo-differential operators $A$ and $B$. As a result we obtain that the KP hierarchy written in the variables $w_\alpha$, $\alpha\ge 1$, has the form
\begin{gather}\label{eq:KP in w}
\frac{\d w_\alpha}{\d T_\beta}=\d_x R_{\alpha\beta},
\end{gather}
where
\begin{align}
&R_{\alpha\beta}\in\cR_{w;\ge 1},\label{eq:property1 of R}\\
&\deg R_{\alpha\beta}=\alpha+\beta,\label{eq:property2 of R}\\
&R_{\alpha,1}=R_{1,\alpha}=w_\alpha,\label{eq:property3 of R}\\
&R_{\alpha\beta}=R_{\beta\alpha}.\label{eq:property4 of R}
\end{align}

\begin{example}
Using Example~\ref{example:L2} we compute
\begin{gather*}
w_1=f_1,\qquad w_2=2f_2+f_1^{(1)},\qquad w_3=3f_3+3f_1^2+3f_2^{(1)}+f_1^{(2)},
\end{gather*}
and 
$$
\frac{\d w_2}{\d T_2}=\d_x\left(\frac{4}{3}w_3-2w_1^2-\frac{1}{3}w_1^{(2)}\right).
$$
\end{example}

\bigskip

\subsection{The main result}

Note that putting $\eps=1$ gives an isomorphism $\hcR^{[0]}_v\stackrel{\cong}{\to}\cR_v$. Therefore, putting $\eps=1$ in the system~\eqref{eq:v-system} we don't lose any information about the equations.\\  

\begin{theorem}\label{theorem:main}
Consider the reduction of the DR hierarchy from Theorem~\ref{theorem:reduction} written in the variables~$v_a$ (the system~\eqref{eq:v-system}) and the KP hierarchy written in the variables~$w_a$ (the system~\eqref{eq:KP in w}). If we put $\eps=1$, then these two systems are related by the change of variables
\begin{gather}\label{eq:DR-KP change of variables}
v_\alpha=-\frac{1}{\alpha}w_\alpha,\qquad t^\beta=\beta T_\beta.
\end{gather}
\end{theorem}

\bigskip

The proof of the theorem is splitted in three steps.\\

\subsubsection{Step 1 of the proof: more properties of the DR hierarchy}\label{subsection:step1}

\begin{lemma}
The polynomials $P_{\alpha\beta}$ satisfy the following properties:
\begin{align}
&P_{\alpha,1}=u_\alpha,&& &&\alpha\ge 1,\label{eq:P1}\\
&P_{\alpha\beta}=u_{\alpha+\beta-1}+\tP_{\alpha\beta}\big(u^{(*)}_{\le\alpha+\beta-3},\eps\big), && \tP_{\alpha\beta}\in\hcR_{u;\ge 1}^{\ev;[0]},&&\alpha,\beta\ge 1,\label{eq:P2}\\
&P_{1,\alpha}=u_\alpha+\eps^2\frac{\alpha(\alpha-2)}{24}u^{(2)}_{\alpha-2}+\eps^2 P'_{1,\alpha}\left(u^{(*)}_{\le\alpha-3},\eps\right),&& P'_{1,\alpha}\in\hcR_{u;\ge 1}^{\ev;[2]},&&\alpha\ge 1,\label{eq:P3}\\
&P_{\alpha,2}=u_{\alpha+1}+\frac{u_1 u_{\alpha-1}}{1+\delta_{\alpha,2}}+\frac{\eps^2}{24}u^{(2)}_{\alpha-1}+P'_{\alpha,2}\left(u^{(*)}_{\le\alpha-2},\eps\right),&& P'_{\alpha,2}\in\hcR_{u;\ge 1}^{\ev;[0]},&&\alpha\ge 1,\label{eq:P4}
\end{align}
where we adopt the convention $u^{(*)}_i:=0$ for $i\le 0$.
\end{lemma}
\begin{proof}
Equation~\eqref{eq:P1} follows from \eqref{eq:polynomials P} where, for $\beta=1$, all the cycles involved in the integral over~$\oM_{g,n+2}$, are pull-backs via the morphism $\pi\colon\oM_{g,n+2}\to\oM_{g,n+1}$ forgetting the second marked point, unless $g=0$ and $n=1$, in which case the integral is over $\oM_{0,3}$ and all the nontrivial cycles involved equal $1$.\\

Equation~\eqref{eq:P2} follows from the fact that, on~$\oM_{0,3}$, all the nontrivial cycles involved in~\eqref{eq:polynomials P} equal $1$.\\

To prove equations~\eqref{eq:P3} and~\eqref{eq:P4}, we have to check that
\begin{align}
&\int_{\oM_{0,4}}\left[\oH^{\res}_0(\alpha-1,1,-2,-\alpha)\right]=1,&&\alpha\ge 2,\label{eq:4-point integral}\\
&\int_{\DR_1(a,0,-a)}\lambda_1\left[\oH^{\res}_1(0,\alpha-1,-\alpha+1)\right]=a^2\frac{\alpha(\alpha-2)}{24},&& \alpha\ge 3,\notag\\
&\int_{\DR_1(a,0,-a)}\lambda_1\left[\oH^{\res}_1(\alpha-1,1,-\alpha)\right]=\frac{a^2}{24},&& \alpha\ge 2.\label{eq:alpha2-integral}
\end{align}
Note that the second equation is equivalent to
\begin{gather}\label{eq:1alpha-integral}
\int_{\oM_{1,2}}\lambda_1\left[\oH^{\res}_1(\alpha-1,-\alpha+1)\right]=\frac{\alpha(\alpha-2)}{24},\qquad \alpha\ge 3,
\end{gather}
where we have used that 
$$
\left[\oH^{\res}_1(0,\alpha-1,-\alpha+1)\right]=\pi^*\left[\oH^{\res}_1(\alpha-1,-\alpha+1)\right],\qquad \pi_* (\lambda_1 \DR_1(a,0,-a))= a^2 \lambda_1,
$$
where $\pi\colon\oM_{1,3}\to \oM_{1,2}$ forgets the first marked point (see, e.g.,~\cite[Lemma~5.4]{BDGR18}). \\

We have two substantially different proofs of equations~\eqref{eq:4-point integral},~\eqref{eq:alpha2-integral},~\eqref{eq:1alpha-integral}, and we think that it is instructive to present both of them.\\

\underline{\it The first proof of equations~\eqref{eq:4-point integral},~\eqref{eq:alpha2-integral},~\eqref{eq:1alpha-integral}}. To prove equation~\eqref{eq:4-point integral}, let us describe the set $\cH^{\res}_0(\alpha-1,1,-2,-\alpha)\subset\cM_{0,4}$ explicitly. The moduli space $\cM_{0,4}$ is isomorphic to $\mbC\setminus\{0,1\}$, with an isomorphism sending a point $t\in\mbC\setminus\{0,1\}$ to the isomorphism class of the marked curve $(\mbCP^1;1,t,0,\infty)$. A unique, up to a multiplicative constant, meromorphic differential on $\mbCP^1$, whose divisor is $(\alpha-1)[1]+[t]-2[0]-\alpha[\infty]$, is given by $\omega=\frac{(z-1)^{\alpha-1}(z-t)}{z^2}dz$. Its residue at $0$ is equal to $(-1)^{\alpha-1}(1+(\alpha-1)t)$. Thus, the differential~$\omega$ is residueless if and only if $t=-\frac{1}{\alpha-1}$. We conclude that $\cH^{\res}_0(\alpha-1,1,-2,-\alpha)\subset\cM_{0,4}$ is a point. Therefore, $\oH^{\res}_0(\alpha-1,1,-2,-\alpha)\subset\oM_{0,4}$ is also a point, which proves~\eqref{eq:4-point integral}.\\

The proof of equations~\eqref{eq:alpha2-integral} and~\eqref{eq:1alpha-integral} is based on the following lemma.

\begin{lemma}\label{lemma:one integral}
We have $\int_{\oM_{1,2}}\psi_1\left[\oH^{\res}_1(a,-a)\right]=\frac{a^2-1}{24}$, $a\ge 1$.
\end{lemma}
\begin{proof}
Consider an arbitrary smooth elliptic curve~$C$ with two marked points~$x_1$ and~$x_2$. Since~$C$ carries a nowhere vanishing holomorphic differential, the fact that there exists a meromorphic differential $\omega$ on $C$ with $(\omega)=a[x_1]-a[x_2]$ is equivalent to the fact that there exists a meromorphic function $f$ on $C$ with $(f)=a[x_1]-a[x_2]$. Therefore, $\left[\oH^{\res}_1(a,-a)\right]$ coincides with the version of the double ramification cycle defined using admissible coverings rather than relative stable maps (see, e.g.,~\cite[Section~2.3]{BSSZ15} and~\cite{Ion02}), which we denote by $\DR_1^{\adm}(a,-a)$. The fact $\int_{\oM_{1,2}}\psi_1\left[\DR^{\adm}_1(a,-a)\right]=\frac{a^2-1}{24}$ follows, for example, from~\cite[Theorem~6]{BSSZ15}.
\end{proof}

\bigskip

For $I\subset[n]$ and $0\le h\le g$ denote by $\delta^I_{h}\in H^2(\oM_{g,n})$ the class of the closure of the substack of stable curves from $\oM_{g,n}$ having exactly one node separating a genus~$h$ component carrying the points marked by~$I$ and the genus~$g-h$ component carrying the points marked by~$[n]\setminus I$.\\

For~\eqref{eq:1alpha-integral} we compute
\begin{align*}
\int_{\oM_{1,2}}\lambda_1&\left[\oH^{\res}_1(\alpha-1,-\alpha+1)\right]=\int_{\oM_{1,2}}\left(\psi_1-\delta_0^{\{1,2\}}\right)\left[\oH^{\res}_1(\alpha-1,-\alpha+1)\right]\stackrel{\substack{\text{Lemma~\ref{lemma:one integral}}\\\text{Proposition~\ref{proposition:partial CohFT Hres}}}}{=}\\
&\hspace{1.5cm}=\frac{\alpha(\alpha-2)}{24}-\left(\int_{\oM_{1,1}}\left[\oH^{\res}_1(0)\right]\right)\left(\int_{\oM_{0,3}}\left[\oH^{\res}_0(-2,\alpha-1,-\alpha+1)\right]\right)=\\
&\hspace{1.5cm}=\frac{\alpha(\alpha-2)}{24},
\end{align*}
where both integrals in the product in the second line vanish because of degree reasons.\\

To prove equation~\eqref{eq:alpha2-integral} we use Hain's formula~\cite[Theorem~11.1]{Hai13}
$$
\DR_1(a,-a)|_{\cM_{1,2}^{\ct}}=a^2\left(\frac{\lambda_1}{2}+\delta_0^{\{1,2\}}\right),
$$
which, together with the fact $\lambda_1^2=0$, gives
\begin{align}
\int_{\DR_1(a,0,-a)}\hspace{-1cm}\lambda_1\left[\oH^{\res}_1(\alpha-1,1,-\alpha)\right]=&a^2\int_{\oM_{1,3}}\hspace{-0.2cm}\lambda_1\left(\delta_0^{\{1,3\}}+\delta_0^{\{1,2,3\}}\right)\left[\oH^{\res}_1(\alpha-1,1,-\alpha)\right]\stackrel{\text{Proposition~\ref{proposition:partial CohFT Hres}}}{=}\notag\\
=&a^2\left(\int_{\oM_{1,1}}\lambda_1\left[\oH^{\res}_1(0)\right]\right)\left(\int_{\oM_{0,4}}\left[\oH^{\res}_0(-2,\alpha-1,1,-\alpha)\right]\right).\label{product of integrals}
\end{align}
Since any smooth elliptic curve carries a nowhere vanishing holomorphic differential, we have $\cH_1^{\res}(0)=\cM_{1,1}$ and, therefore, $\left[\oH_1^{\res}(0)\right]=1\in H^0(\oM_{1,1})$. Since $\int_{\oM_{1,1}}\lambda_1=\frac{1}{24}$, the expression in line~\eqref{product of integrals} is equal to $\frac{a^2}{24}\int_{\oM_{0,4}}\left[\oH^{\res}_0(-2,\alpha-1,1,-\alpha)\right]=\frac{a^2}{24}$ by~\eqref{eq:4-point integral}. \\

\underline{\it The second proof of equations~\eqref{eq:4-point integral},~\eqref{eq:alpha2-integral},~\eqref{eq:1alpha-integral}}. Equation~\eqref{eq:4-point integral} follows from \cite[Propositions~8.2 and~8.3]{CMZ20}, based in turn on \cite[Theorem 6(1),(3)]{Sau19}, where, for $g,n,k\geq 0$, $m\geq 2$, such that $2g-2+n+m+k>0$, and integers $a_1,\ldots, a_n\geq 0$, $b_1,\ldots,b_m \geq 1$, $c_1,\ldots,c_k\geq 2$, the authors computed the class of the moduli stack
$$
\oH_g(a_1,\ldots,a_n,-b_1,\ldots,-b_m;-c_1,\ldots,-c_k)
$$
inside the moduli stack of projectivized meromorphic differentials with one less residue condition,
$$
\oH_g(a_1,\ldots,a_n,-b_1,\ldots,-b_{m-1};-b_m,-c_1,\ldots,-c_k),
$$
as a linear combination of psi classes and boundary divisors. According to that formula, $\left[\oH_0^{\res}(\alpha-1,1,-2,-\alpha)\right]=(\alpha-1)\psi_4-(\alpha-2)\delta^{\{1,3\}}_0=\psi_4$, which immediately yields the desired result.\\

Equations~\eqref{eq:alpha2-integral} and~\eqref{eq:1alpha-integral} follow from~\cite[equation (31)]{FP18}, which, for $a_1,\ldots,a_n\in \mbZ$ with at least one negative entry, computes the discrepancy between the class $[\oH_g(a_1,\ldots,a_n)]$ and the weighted fundamental class $\mathsf{H}_g(a_1\dots,a_n)$ of the moduli space of twisted canonical divisors $\widetilde{\cH}_g(a_1,\ldots,a_n)$. As, by the results of \cite{BHPSS20}, $\mathsf{H}_g(a_1\dots,a_n)$ equals the $1$-twisted DR cycle $\DR_g^1(a_1,\ldots,a_n)$ of \cite{JPPZ17}, in particular one obtains
\begin{align*}
&[\oH_1(\alpha-1,-\alpha+1)]= \DR_1^1(\alpha-1,-\alpha+1) - \delta_0^{\{1,2\}}, & \alpha\ge 3,\\
&[\oH_1(\alpha-1,1,-\alpha)]= \DR_1^1(\alpha-1,1,-\alpha) - \delta_0^{\{1,2,3\}}, & \alpha\geq 2.
\end{align*} 
Since the $1$-twisted DR cycle $\DR_1^1(a_1,\ldots,a_n)$ equals the untwisted DR cycle $\DR_1(a_1,\ldots,a_n)$ in genus $1$ via geometric arguments, a simple application of Hain's formula yields both desired results.
\end{proof}

\bigskip

\begin{example}
The lemma fully determines several polynomials $P_{\alpha\beta}$:
\begin{gather*}
P_{1,2}=u_2,\qquad P_{1,3}=u_3+\frac{\eps^2}{8}u_1^{(2)},\qquad P_{2,2}=u_3+\frac{u_1^2}{2}+\frac{\eps^2}{24}u_1^{(2)}.
\end{gather*}
\end{example}

\bigskip

Recall that the polynomials $Q_{\alpha\beta}$ satisfy the following properties:
\begin{align}
&Q_{\alpha\beta}\in\hcR_{v;\ge 1}^{\ev;[0]},\label{eq:Q1}\\
&\deg Q_{\alpha\beta}=\alpha+\beta,\label{eq:Q2}\\
&Q_{\alpha,1}=Q_{1,\alpha}=v_\alpha,\label{eq:Q3}\\
&Q_{\alpha\beta}=Q_{\beta\alpha}.\label{eq:Q4}
\end{align}
The lemma implies that we also have 
\begin{align}
&Q_{\alpha\beta}=v_{\alpha+\beta-1}+\tQ_{\alpha\beta}\big(v^{(*)}_{\le\alpha+\beta-3},\eps\big), && \tQ_{\alpha\beta}\in\hcR_{v;\ge 1}^{\ev;[0]},\label{eq:Q5}\\
&Q_{\alpha,2}=v_{\alpha+1}+\frac{v_1 v_{\alpha-1}}{1+\delta_{\alpha,2}}-\frac{\alpha-1}{12}v^{(2)}_{\alpha-1}\eps^2+Q'_{\alpha,2}\big(v^{(*)}_{\le\alpha-2},\eps\big), && Q'_{\alpha,2}\in\hcR^{\ev;[0]}_{v;\ge 1}.\label{eq:Q6}
\end{align}

\bigskip 

\subsubsection{Step 2 of the proof: more properties of the KP hierarchy}\label{subsection:step2}

\begin{lemma}\label{lemma:KP1}
We have $R_{\alpha\beta}\in\cR^{\ev}_{w;\ge 1}$.
\end{lemma}
\begin{proof}
There is an involution on the space of pseudo-differential operators given by 
$$
\left(\sum_{n=-\infty}^m a_n\d_x^n\right)^\dagger:=\sum_{n=-\infty}^m (-\d_x)^n\circ a_n.
$$
It satisfies the properties $(A\circ B)^\dagger=B^\dagger\circ A^\dagger$ and $\res A^\dagger=-\res A$ for any two pseudo-differential operators $A$ and $B$.\\

Consider the change of variables $f_i\mapsto \tf_i(f^{(*)}_*)$ given by
\begin{align*}
&L=\d_x+\sum_{i\ge 1}f_i\d_x^{-i}\,\mapsto\\
\mapsto\,&\tL=\d_x+\sum_{i\ge 1}\tf_i(f^{(*)}_*)\d_x^{-i}:=-L^\dagger=\d_x+f_1\d_x^{-1}+(-f_2-f_1^{(1)})\d_x^{-2}+(f_3+2f_2^{(1)}+f_1^{(2)})\d_x^{-3}+\ldots.
\end{align*}
It is clearly invertible and it induces a change of variables $w_\alpha\mapsto \tw_\alpha(w^{(*)}_*)$, for which we compute
$$
\tw_\alpha(w^{(*)}_*)=\res\tL^a=(-1)^a\res(L^\dagger)^a=(-1)^a\res(L^a)^\dagger=(-1)^{a+1}\res L^a=(-1)^{a+1} w_a.
$$
Therefore, the KP hierarchy written in the variables $\tw_\alpha$ has the form
\begin{gather}\label{eq:KP in tw}
\frac{\d\tw_\alpha}{\d T_\beta}=\d_x\tR_{\alpha\beta},\qquad\tR_{\alpha\beta}=(-1)^{\alpha+1}\left.R_{\alpha\beta}\right|_{w_\gamma^{(k)}=(-1)^{\gamma+1}\tw_\gamma^{(k)}}\in\cR_{\tw;\ge 1}.
\end{gather}

\bigskip

On the other hand, we compute
$$
\frac{\d\tL}{\d T_\beta}=-\left(\frac{\d L}{\d T_\beta}\right)^\dagger=-[(L^\beta)_+,L]^\dagger=\left[\left((L^\dagger)^\beta\right)_+,L^\dagger\right]=(-1)^{\beta+1}\left[(\tL^\beta)_+,\tL\right],
$$
and, therefore, $\frac{\d\tw_\alpha}{\d T_\beta}=(-1)^{\beta+1}\res\left[(\tL^\beta)_+,\tL^\alpha\right]$. Hence, $\tR_{\alpha\beta}=\left.(-1)^{\beta+1}R_{\alpha\beta}\right|_{w^{(k)}_\gamma=\tw^{(k)}_\gamma}$. Combining this with~\eqref{eq:KP in tw} we obtain $\left.(-1)^{\alpha+\beta}R_{\alpha\beta}\right|_{w^{(k)}_\gamma\mapsto(-1)^{\gamma+1}w^{(k)}_\gamma}=R_{\alpha\beta}$. Together with the property $\deg R_{\alpha\beta}=\alpha+\beta$ this implies that $\left.R_{\alpha\beta}\right|_{w^{(k)}_\gamma\mapsto(-1)^k w^{(k)}_\gamma}=R_{\alpha\beta}$, which gives $R_{\alpha\beta}\in\cR_w^\ev$, as required.
\end{proof}

\bigskip

\begin{lemma}\label{lemma:KP2}
Let $k\ge 1$.
\begin{enumerate}[ 1.]
\item The coefficients of the pseudo-differential operator $L^k-\d_x^k-\sum_{i\ge 1}\sum_{l=0}^{k-1}{k\choose l}f_i^{(k-1-l)}\d_x^{-i+l}$ belong to the ring $\cR_{f;\ge 2}$.

\medskip

\item $\displaystyle S_{i,k}=\sum_{j=1}^k{k \choose j} f_{i+k-j}^{(j)}+\tS_{i,k}\big(f^{(*)}_{\le i+k-3}\big)$, where $\tS_{i,k}\in\cR_{f;\ge 2}$.

\medskip

\item $\displaystyle w_k(f^{(*)}_*)=\sum_{i=0}^{k-1}{k\choose k-1-i}f_{k-i}^{(i)}+\frac{k(k-1)}{1+\delta_{k,3}}f_1 f_{k-2}+T_k\big(f_{\le k-3}^{(*)}\big)$, where $T_k\in\cR_{f;\ge 2}$.

\medskip

\item $\displaystyle f_k(w^{(*)}_*)=
\begin{cases}
\frac{1}{k}\sum_{j=0}^{k-1}{k\choose j}\cB_j w_{k-j}^{(j)},&\text{if $k\le 2$},\\
\frac{1}{k}\sum_{j=0}^{k-1}{k\choose j}\cB_j w_{k-j}^{(j)}-\frac{1}{1+\delta_{k,3}}\frac{k-1}{k-2}w_1w_{k-2}+K_k\big(w^{(*)}_{\le k-3}\big),&\text{if $k\ge 3$},
\end{cases}
$, where $K_k\in\cR_{w;\ge 2}$ and we recall that $\cB_j$ are the Bernoulli numbers.
\end{enumerate}
\end{lemma}
\begin{proof}
{\it 1}. This can be easily proved by induction.\\

{\it 2}. Using the first part we see that, up to terms from~$\cR_{f;\ge 2}$, the coefficient of $\d_x^{-i}$, $i\ge 1$, in $[(L^k)_+,L]$ is equal to the coefficient of $\d_x^{-i}$ in $[\d_x^k,\sum_{j\ge 1}f_j\d_x^{-j}]$, from which we get the required formula for $S_{i,k}$.\\

{\it 3}. The formula for the linear part of $w_k(f^{(*)}_*)=\res L^k$ immediately follows from the first part of the lemma. In order to determine the coefficient of~$f_1f_{k-2}$, for $k\ge 3$, we compute 
$$
\frac{\d\res L^k}{\d f_{k-2}}=\sum_{a+b=k-1}\res\left(L^a\circ\d_x^{-k+2}\circ L^b\right)=k(k-1)f_1.
$$

\bigskip

{\it 4}. The formula for the linear part of~$f_k(w^{(*)}_*)$ follows from the previous part and the standard property of the Bernoulli numbers: $\sum_{j=0}^a{a+1\choose j}\cB_j=\delta_{a,0}$, $a\ge 0$. The coefficient of $w_1w_{k-2}$ is found from the previous part by an elementary computation.
\end{proof}

\bigskip

The last two lemmas imply that
$$
R_{\alpha\beta}=\frac{\alpha\beta}{\alpha+\beta-1}w_{\alpha+\beta-1}+\widetilde{R}_{\alpha\beta}\big(w^{(*)}_{\le\alpha+\beta-3}\big),\qquad \tR_{\alpha\beta}\in\cR^\ev_{w;\ge 1}.
$$

\bigskip

\begin{lemma}\label{lemma:KP3}
{\ }
\begin{enumerate}[ 1.]
\item For $k\ge 1$ we have $\displaystyle S_{k,2}=2f_{k+1}^{(1)}+f_k^{(2)}+2(k-1)f_{k-1}f_1^{(1)}+S_{k,2}'\big(f^{(*)}_{\le k-2}\big)$, where $S'_{k,2}\in\cR_{f;\ge 2}$.

\medskip

\item For $k\ge 2$ we have $\displaystyle R_{k,2}=\frac{2k}{k+1}w_{k+1}-\frac{1}{1+\delta_{k,2}}\frac{2k}{k-1}w_1 w_{k-1}-\frac{k}{6} w^{(2)}_{k-1}+R'_{k,2}\big(w^{(*)}_{\le k-2}\big)$, where $R'_{k,2}\in\cR^\ev_{w;\ge 1}$. 
\end{enumerate}
\end{lemma}
\begin{proof}
{\it 1}. From Lemma~\ref{lemma:KP2} and the property $\deg S_{k,2}=k+2$ we conclude that 
$$
S_{k,2}=2f_{k+1}^{(1)}+f_k^{(2)}+\alpha f_1 f_{k-1}^{(1)}+\beta f_1^{(1)}f_{k-1}+S_{k,2}'\big(f^{(*)}_{\le k-2}\big),\qquad S'_{k,2}\in\cR_{f;\ge 2}.
$$
In order to determine $\alpha$ and $\beta$ we compute
\begin{align*}
&\frac{\d S_{k,2}}{\d f_1}=\Coef_{\d_x^{-k}}\frac{\d}{\d f_1}[L^2_+,L]=\Coef_{\d_x^{-k}}[\d_x^2+2f_1,\d_x^{-1}]=2(-1)^k f_1^{(k-1)}, && k\ge 2,\\
&\frac{\d S_{k,2}}{\d f_{k-1}}=\Coef_{\d_x^{-k}}\frac{\d}{\d f_{k-1}}[L^2_+,L]=\Coef_{\d_x^{-k}}[\d_x^2+2f_1,\d_x^{-(k-1)}]=2(k-1)f_1^{(1)}, && k\ge 2,
\end{align*}
which implies the required formula for $S_{k,2}$.\\

{\it 2}. This is an elementary computation based on the first part and Lemma~\ref{lemma:KP2}.
\end{proof}

\bigskip

Summarizing our computations with the KP hierarchy, we have
\begin{align}
&R_{\alpha\beta}\in\cR^\ev_{w;\ge 1},&& &&   \label{eq:R1}\\
&\deg R_{\alpha\beta}=\alpha+\beta,&& &&     \label{eq:R2}\\
&R_{\alpha,1}=R_{1,\alpha}=w_\alpha,&& &&    \label{eq:R3}\\
&R_{\alpha\beta}=R_{\beta\alpha},&& &&       \label{eq:R4}
\end{align}
\begin{align}
&R_{\alpha\beta}=\frac{\alpha\beta}{\alpha+\beta-1}w_{\alpha+\beta-1}+\widetilde{R}_{\alpha\beta}\big(w^{(*)}_{\le\alpha+\beta-3}\big), && \tR_{\alpha\beta}\in\cR^\ev_{w;\ge 1}, && \label{eq:R5}\\
&R_{\alpha,2}=\frac{2\alpha\,w_{\alpha+1}}{\alpha+1}-\frac{2\alpha}{\alpha-1}\frac{w_1 w_{\alpha-1}}{1+\delta_{\alpha,2}}-\frac{\alpha}{6} w^{(2)}_{\alpha-1}\eps^2+R'_{\alpha,2}\big(w^{(*)}_{\le \alpha-2}\big), && R'_{\alpha,2}\in\cR^\ev_{w;\ge 1}, && \alpha\ge 2.\label{eq:R6} 
\end{align}

\bigskip

\subsubsection{Step 3 of the proof: a limited amount of data determines the hierarchies uniquely}\label{subsection:step3}

It is clear that the change of variables~\eqref{eq:DR-KP change of variables} (together with putting $\eps=1$) transforms the properties~\eqref{eq:Q1}--\eqref{eq:Q6} of the system~\eqref{eq:v-system} exactly to the properties~\eqref{eq:R1}--\eqref{eq:R6} of the system~\eqref{eq:KP in w}. Thus, the following theorem will complete the proof of Theorem~\ref{theorem:main}.\\

\begin{theorem}\label{theorem:reconstruction}
The commutativity of the flows $\frac{\d}{\d t^\alpha}$ together with the properties \eqref{eq:Q1}--\eqref{eq:Q6} determines all the polynomials $Q_{\alpha\beta}$ uniquely.
\end{theorem}
\begin{proof}
We start with the following lemma.

\begin{lemma}
For any $\alpha,\beta\ge 1$ we have the following relation:
\begin{gather}\label{eq:relation for Q}
\d_x Q_{\alpha+1,\beta}=\d_x Q_{\alpha+\beta-1,2}+\sum_{i=1}^{\alpha+\beta-3}\sum_{j\ge 0}\frac{\d\tQ_{\alpha\beta}}{\d v^{(j)}_i}\d_x^{j+1}Q_{i,2}-\sum_{i=1}^{\alpha-1}\sum_{j\ge 0}\frac{\d\tQ_{\alpha,2}}{\d v^{(j)}_i}\d_x^{j+1}Q_{i,\beta}.
\end{gather}
\end{lemma}
\begin{proof}
The relation $\frac{\d}{\d t^2}\frac{\d v_\alpha}{\d t^\beta}=\frac{\d}{\d t^\beta}\frac{\d v_\alpha}{\d t^2}$ gives $\frac{\d}{\d t^2}\left(v^{(1)}_{\alpha+\beta-1}+\d_x\tQ_{\alpha\beta}\right)=\frac{\d}{\d t^\beta}\left(v^{(1)}_{\alpha+1}+\d_x\tQ_{\alpha,2}\right)$, which immediately implies~\eqref{eq:relation for Q}.
\end{proof}

Note that if $\alpha+\beta+1=d$, then the right-hand side of~\eqref{eq:relation for Q} contains only the polynomial~$Q_{d-2,2}$ together with the polynomials $Q_{\gamma\delta}$ with $\gamma+\delta\le d-1$. Therefore, relation~\eqref{eq:relation for Q} determines recursively all the polynomials $Q_{\alpha\beta}$ with $\alpha,\beta\ge 3$ starting from the polynomials~$Q_{\gamma,2}$.\\ 

We now have to show how to reconstruct the polynomials $Q_{\alpha,2}$, $\alpha\ge 2$, starting from the polynomial~$Q_{2,2}$, which, by~\eqref{eq:Q6}, is equal to
\begin{gather}\label{eq:Q22}
Q_{2,2}=v_3+\frac{v_1^2}{2}-\frac{\eps^2}{12}v_1^{(2)}.
\end{gather}
Let $\beta\ge 4$ and let us write relation~\eqref{eq:relation for Q} for $\alpha=2$:
\begin{align}
&\d_x Q_{3,\beta}=\d_x Q_{\beta+1,2}+\sum_{i=1}^{\beta-1}\sum_{j\ge 0}\frac{\d\tQ_{2,\beta}}{\d v^{(j)}_i}\d_x^{j+1}Q_{i,2}-\sum_{j\ge 0}\frac{\d\tQ_{2,2}}{\d v^{(j)}_1}v_\beta^{(j+1)}\,\stackrel{\text{\eqref{eq:Q22}}}{\Rightarrow}\notag\\
\Rightarrow\,& \d_x Q_{3,\beta}=\d_x Q_{\beta+1,2}+\sum_{i=1}^{\beta-1}\sum_{j\ge 0}\frac{\d\tQ_{2,\beta}}{\d v^{(j)}_i}\d_x^{j+1}Q_{i,2}-v_1 v_{\beta}^{(1)}+\frac{\eps^2}{12}v_\beta^{(3)}.\label{eq:relation1}
\end{align}
On the other hand, relation~\eqref{eq:relation for Q} also gives
\begin{gather}\label{eq:relation2}
\d_x Q_{\beta,3}=\d_x Q_{\beta+1,2}+\sum_{i=1}^{\beta-1}\sum_{j\ge 0}\frac{\d\tQ_{\beta-1,3}}{\d v^{(j)}_i}\d_x^{j+1}Q_{i,2}-\sum_{i=1}^{\beta-2}\sum_{j\ge 0}\frac{\d\tQ_{\beta-1,2}}{\d v^{(j)}_i}\d_x^{j+1}Q_{i,3}.
\end{gather}
Equating the right-hand sides of equations~\eqref{eq:relation1} and~\eqref{eq:relation2}, and cancelling the terms $\d_x Q_{\beta+1,2}$, we obtain
$$
\sum_{i=1}^{\beta-1}\sum_{j\ge 0}\frac{\d\tQ_{2,\beta}}{\d v^{(j)}_i}\d_x^{j+1}Q_{i,2}-v_1 v_{\beta}^{(1)}+\frac{\eps^2}{12}v_\beta^{(3)}=\sum_{i=1}^{\beta-1}\sum_{j\ge 0}\frac{\d\tQ_{\beta-1,3}}{\d v^{(j)}_i}\d_x^{j+1}Q_{i,2}-\sum_{i=1}^{\beta-2}\sum_{j\ge 0}\frac{\d\tQ_{\beta-1,2}}{\d v^{(j)}_i}\d_x^{j+1}Q_{i,3}.
$$
Using again relation~\eqref{eq:relation1} in order to express $\tQ_{\beta-1,3}=\tQ_{3,\beta-1}$ and $Q_{i,3}=Q_{3,i}$ in terms of the differential polynomials $Q_{\gamma,2}$, we obtain
\begin{align*}
&\underline{\sum_{i=1}^{\beta-1}\sum_{j\ge 0}\frac{\d\tQ_{2,\beta}}{\d v^{(j)}_i}\d_x^{j+1}Q_{i,2}}-v_1 v_{\beta}^{(1)}+\frac{\eps^2}{12}v_\beta^{(3)}=\\
=&\d_x^{-1}\left[\sum_{i=1}^{\beta-1}\sum_{j\ge 0}\frac{\d}{\d v^{(j)}_i}\left(\underline{\d_x\tQ_{\beta,2}}+\sum_{k=1}^{\beta-2}\sum_{l\ge 0}\frac{\d\tQ_{2,\beta-1}}{\d v^{(l)}_k}\d_x^{l+1}Q_{k,2}-v_1 v_{\beta-1}^{(1)}+\frac{\eps^2}{12}v_{\beta-1}^{(3)}\right)\d_x^{j+1}Q_{i,2}\right]\\
&-\sum_{i=1}^{\beta-2}\sum_{j\ge 0}\frac{\d\tQ_{\beta-1,2}}{\d v^{(j)}_i}\d_x^{j}\left(\d_x Q_{i+1,2}+\sum_{k=1}^{i-1}\sum_{l\ge 0}\frac{\d\tQ_{2,i}}{\d v^{(l)}_k}\d_x^{l+1}Q_{k,2}-v_1 v_{i}^{(1)}+\frac{\eps^2}{12}v_i^{(3)}\right),
\end{align*}
which, cancelling the underlined terms, is equivalent to
\begin{align*}
&-v_1 v_{\beta}^{(1)}+\frac{\eps^2}{12}v_\beta^{(3)}=\\
=&\d_x^{-1}\left[\sum_{i=1}^{\beta-1}\sum_{k=1}^{\beta-2}\sum_{j,l\ge 0}\frac{\d}{\d v^{(j)}_i}\left(\frac{\d\tQ_{2,\beta-1}}{\d v^{(l)}_k}\d_x^{l+1}Q_{k,2}\right)\d_x^{j+1}Q_{i,2}-v_2^{(1)} v_{\beta-1}^{(1)}-v_1 \d_x^2 Q_{\beta-1,2}+\frac{\eps^2}{12}\d_x^4 Q_{\beta-1,2}\right]\\
&-\sum_{i=1}^{\beta-2}\sum_{j\ge 0}\frac{\d\tQ_{\beta-1,2}}{\d v^{(j)}_i}\d_x^{j}\left(\d_x Q_{i+1,2}+\sum_{k=1}^{i-1}\sum_{l\ge 0}\frac{\d\tQ_{2,i}}{\d v^{(l)}_k}\d_x^{l+1}Q_{k,2}-v_1 v_{i}^{(1)}+\frac{\eps^2}{12}v_i^{(3)}\right).
\end{align*}

\bigskip

Splitting the two summations over $i$ and collecting $\frac{\eps^2}{12}v_\beta^{(3)}-\frac{\eps^2}{12}\d_x^3 Q_{\beta-1,2}=-\frac{\eps^2}{12}\d_x^3\tQ_{\beta-1,2}$, we obtain 
\begin{align*}
&-v_1 v_{\beta}^{(1)}-\frac{\eps^2}{12}\d_x^3\tQ_{\beta-1,2}=\\
=&\d_x^{-1}\Bigg[\sum_{i,k=1}^{\beta-2}\sum_{j,l\ge 0}\frac{\d}{\d v^{(j)}_i}\left(\frac{\d\tQ_{2,\beta-1}}{\d v^{(l)}_k}\d_x^{l+1}Q_{k,2}\right)\d_x^{j+1}Q_{i,2}+\underbrace{\sum_{k=1}^{\beta-2}\sum_{j,l\ge 0}\frac{\d}{\d v^{(j)}_{\beta-1}}\left(\frac{\d\tQ_{2,\beta-1}}{\d v^{(l)}_k}\d_x^{l+1}Q_{k,2}\right)\d_x^{j+1}Q_{\beta-1,2}}_{E:=}\\
&-v_2^{(1)} v_{\beta-1}^{(1)}-v_1 \d_x^2 Q_{\beta-1,2}\Bigg]\\
&-\sum_{i=1}^{\beta-3}\sum_{j\ge 0}\frac{\d\tQ_{\beta-1,2}}{\d v^{(j)}_i}\d_x^{j}\left(\d_x Q_{i+1,2}+\sum_{k=1}^{i-1}\sum_{l\ge 0}\frac{\d\tQ_{2,i}}{\d v^{(l)}_k}\d_x^{l+1}Q_{k,2}-v_1 v_{i}^{(1)}+\frac{\eps^2}{12}v_i^{(3)}\right)\\
&-\underbrace{\sum_{j\ge 0}\frac{\d\tQ_{\beta-1,2}}{\d v^{(j)}_{\beta-2}}\d_x^{j}\left(\d_x Q_{\beta-1,2}+\sum_{k=1}^{\beta-3}\sum_{l\ge 0}\frac{\d\tQ_{2,\beta-2}}{\d v^{(l)}_k}\d_x^{l+1}Q_{k,2}-v_1 v_{\beta-2}^{(1)}+\frac{\eps^2}{12}v_{\beta-2}^{(3)}\right)}_{F:=}.
\end{align*}
From this, computing $E$ and $F$ using formula~\eqref{eq:Q6},
\begin{align*}
E=&\sum_{k=1}^{\beta-2}\sum_{j,l\ge 0}\frac{\d\tQ_{2,\beta-1}}{\d v^{(l)}_k}\frac{\d(\d_x^{l+1}Q_{k,2})}{\d v^{(j)}_{\beta-1}}\d_x^{j+1}Q_{\beta-1,2}=\sum_{l\ge 0}\frac{\d\tQ_{2,\beta-1}}{\d v^{(l)}_{\beta-2}}\d_x^{l+2}Q_{\beta-1,2}=\\
=&v_1\d_x^2 Q_{\beta-1,2}-\eps^2\frac{\beta-2}{12}\d_x^4 Q_{\beta-1,2},
\end{align*}
\begin{align*}
F=&v_1\left(\d_x Q_{\beta-1,2}+\sum_{k=1}^{\beta-3}\sum_{l\ge 0}\frac{\d\tQ_{2,\beta-2}}{\d v^{(l)}_k}\d_x^{l+1}Q_{k,2}-v_1 v_{\beta-2}^{(1)}+\frac{\eps^2}{12}v_{\beta-2}^{(3)}\right)\\
&-\eps^2\frac{\beta-2}{12}\d_x^2\left(\d_x Q_{\beta-1,2}+\sum_{k=1}^{\beta-3}\sum_{l\ge 0}\frac{\d\tQ_{2,\beta-2}}{\d v^{(l)}_k}\d_x^{l+1}Q_{k,2}-v_1 v_{\beta-2}^{(1)}+\frac{\eps^2}{12}v_{\beta-2}^{(3)}\right),
\end{align*}
we obtain
\begin{align}
&\d_x^{-1}\Bigg[\sum_{i,k=1}^{\beta-2}\sum_{j,l\ge 0}\frac{\d}{\d v^{(j)}_i}\left(\frac{\d\tQ_{2,\beta-1}}{\d v^{(l)}_k}\d_x^{l+1}Q_{k,2}\right)\d_x^{j+1}Q_{i,2}-v_2^{(1)} v_{\beta-1}^{(1)}\Bigg]\label{eq:main relation for Qbeta2}\\
&-\sum_{i=1}^{\beta-3}\sum_{j\ge 0}\frac{\d\tQ_{\beta-1,2}}{\d v^{(j)}_i}\d_x^{j}\left(\d_x Q_{i+1,2}+\sum_{k=1}^{i-1}\sum_{l\ge 0}\frac{\d\tQ_{2,i}}{\d v^{(l)}_k}\d_x^{l+1}Q_{k,2}-v_1 v_{i}^{(1)}+\frac{\eps^2}{12}v_i^{(3)}\right)\notag\\
&-v_1\left(\d_x Q_{\beta-1,2}+\sum_{k=1}^{\beta-3}\sum_{l\ge 0}\frac{\d\tQ_{2,\beta-2}}{\d v^{(l)}_k}\d_x^{l+1}Q_{k,2}-v_1 v_{\beta-2}^{(1)}+\frac{\eps^2}{12}v_{\beta-2}^{(3)}\right)\notag\\
&+\eps^2\frac{\beta-2}{12}\d_x^{2}\left(\sum_{k=1}^{\beta-3}\sum_{l\ge 0}\frac{\d\tQ_{2,\beta-2}}{\d v^{(l)}_k}\d_x^{l+1}Q_{k,2}-v_1 v_{\beta-2}^{(1)}+\frac{\eps^2}{12}v_{\beta-2}^{(3)}\right)+v_1 v_{\beta}^{(1)}+\frac{\eps^2}{12}\d_x^3\tQ_{\beta-1,2}=0.\notag
\end{align}
In the rest of the proof, we will show how to use this relation in order to determine all the polynomials $Q_{\gamma,2}$, $\gamma\ge 1$.

\bigskip

For any $\gamma\ge 1$ introduce a polynomial $r_\gamma(v_1,\ldots,v_{\gamma-1})\in\cR_v$, $\tdeg r_\gamma=2$, as follows:
$$
Q_{\gamma,2}=v_{\gamma+1}+r_\gamma+\left(\text{monomials of $\tdeg\ge 3$}\right)+O(\eps^2).
$$
\begin{lemma}
We have $r_\gamma=\frac{1}{2}\sum_{i+k=\gamma}v_i v_k$.
\end{lemma}
\begin{proof}
We already know this for $\gamma=1,2$, so we need to prove it for $\gamma\ge 3$. Consider equation~\eqref{eq:main relation for Qbeta2}, where we recall that $\beta\ge 4$. Let $r_{i,k}:=\left.\frac{\d^2 Q_{\beta-1,2}}{\d v_i\d v_k}\right|_{v_*=0}$. Note that $r_{i,k}=r_{k,i}$, and that $r_{i,k}=0$ unless $i+k=\beta-1$. We know that $r_{1,\beta-2}=1$. Equation~\eqref{eq:main relation for Qbeta2} in particular means that
\begin{align*}
& \int\left(\sum_{i,k=1}^{\beta-2}\sum_{j,l\ge 0}\frac{\d}{\d v^{(j)}_i}\left(\frac{\d\tQ_{2,\beta-1}}{\d v^{(l)}_k}\d_x^{l+1}Q_{k,2}\right)\d_x^{j+1}Q_{i,2}-v_2^{(1)} v_{\beta-1}^{(1)}\right)dx=0\quad\Rightarrow\\
\Rightarrow\quad & \int\left(\sum_{i,k=1}^{\beta-2}\sum_{j\ge 0}\frac{\d}{\d v^{(j)}_i}\left(\frac{\d r_{\beta-1}}{\d v_k}v^{(1)}_{k+1}\right)v^{(j+1)}_{i+1}-v_2^{(1)} v_{\beta-1}^{(1)}\right)dx=0. 
\end{align*}
The last integral is equal to 
\begin{align*}
&\int\left(\sum_{i,k=1}^{\beta-2}\frac{\d^2 r_{\beta-1}}{\d v_i\d v_k}v_{i+1}^{(1)}v^{(1)}_{k+1}+\sum_{k=1}^{\beta-3}\frac{\d r_{\beta-1}}{\d v_k}v^{(2)}_{k+2}-v_2^{(1)} v_{\beta-1}^{(1)}\right)dx=\\
=&\int\left(\sum_{i=1}^{\beta-2}\sum_{k=1}^{\beta-3}r_{i,k}\left(v_{i+1}^{(1)}v^{(1)}_{k+1}-v_i^{(1)}v_{k+2}^{(1)}\right)\right)dx.
\end{align*}

\bigskip

Note that if for a quadratic polynomial $p$ in the variables $v_1^{(1)},\ldots,v_{\beta-2}^{(1)}$ we have $\int p dx=0$, then $p=0$. Therefore, we have
$$
0=\sum_{i=1}^{\beta-2}\sum_{k=1}^{\beta-3}r_{i,k}\left(v_{i+1}^{(1)}v^{(1)}_{k+1}-v_i^{(1)}v_{k+2}^{(1)}\right)=\sum_{i,k=2}^{\beta-3}(r_{i,k}-r_{i+1,k-1})v_{i+1}^{(1)}v^{(1)}_{k+1}+(r_{\beta-2,1}-r_{2,\beta-3})v_2^{(1)}v_{\beta-1}^{(1)},
$$
which implies 
\begin{align*}
&r_{1,\beta-2}=r_{2,\beta-3},\\
&r_{i+1,k-1}+r_{i-1,k+1}=2r_{i,k},\qquad 2\le i,k\le\beta-3,\quad i+k=\beta-1.
\end{align*}
Since $r_{1,\beta-2}=1$, this immediately gives that $r_{i,k}=1$ for $i+k=\beta-1$, as required.
\end{proof}

\bigskip

Consider relation~\eqref{eq:main relation for Qbeta2} and suppose that we know the polynomials $Q_{\gamma,2}$ for $\gamma\le \beta-2$. Then equation~\eqref{eq:main relation for Qbeta2} can be considered as a linear equation for the polynomial $Q_{\beta-1,2}$. Let us show that it has a unique solution (assuming of course that the properties~\eqref{eq:Q1}--\eqref{eq:Q6} are satisfied). This would determine all the polynomials~$Q_{\gamma,2}$ step by step starting from $Q_{2,2}=v_3+\frac{v_1^2}{2}-\frac{\eps^2}{12}v_1^{(2)}$. Suppose that equation~\eqref{eq:main relation for Qbeta2} has two solutions $Q_{\beta-1,2}\ne \widehat{Q}_{\beta-1,2}$. Then, if we denote $R:=Q_{\beta-1,2}-\widehat{Q}_{\beta-1,2}\ne 0$, the expression
\begin{align}
&\Bigg[\sum_{i,k=1}^{\beta-2}\sum_{j,l\ge 0}\frac{\d}{\d v^{(j)}_i}\left(\frac{\d R}{\d v^{(l)}_k}\d_x^{l+1}Q_{k,2}\right)\d_x^{j+1}Q_{i,2}\Bigg]\label{eq:function of R}\\
&+\d_x\left[-\sum_{i=1}^{\beta-3}\sum_{j\ge 0}\frac{\d R}{\d v^{(j)}_i}\d_x^{j}\left(\d_x Q_{i+1,2}+\sum_{k=1}^{i-1}\sum_{l\ge 0}\frac{\d\tQ_{2,i}}{\d v^{(l)}_k}\d_x^{l+1}Q_{k,2}-v_1 v_{i}^{(1)}+\frac{\eps^2}{12}v_i^{(3)}\right)-v_1 \d_x R+\frac{\eps^2}{12}\d_x^3 R\right]\notag
\end{align}
vanishes. Let us decompose $R=R_{2g}\eps^{2g}+O(\eps^{2g+2})$, where $g\ge 0$ and $R_{2g}\ne 0$. Let us further decompose $R_{2g}=A+B$, where $A\ne 0$, $\tdeg A=d\ge 1$, and $B\in\cR_{v;\ge d+1}$.\\

{\it \underline{Case 1: $d=1$}}. Since $Q_{\beta-1,2}$ and $\widehat{Q}_{\beta-1,2}$ have the form~\eqref{eq:Q6}, we have $g\ge 2$. Let us express the polynomial $R$ as follows:
$$
R=\left(\lambda v_{\beta-2g}^{(2g)}+\Omega+\left(\text{monomials of $\tdeg\ge 3$}\right)\right)\eps^{2g}+O(\eps^{2g+2}),\quad g\ge 2,\quad \beta\ge 2g+1,\quad\lambda\ne 0,
$$
where 
$$
\Omega=\frac{1}{2}\sum_{i=1}^{\beta-2g-2}\sum_{j=0}^{2g}\omega_{i,j}v_i^{(j)}v_{\beta-2g-1-i}^{(2g-j)},\quad \omega_{i,j}=\omega_{\beta-2g-1-i,2g-j}.
$$
Then the expression~\eqref{eq:function of R} has the form $\eps^{2g}(C+D)+O(\eps^{2g+2})$, where
\begin{align*}
C=&\sum_{i,k=1}^{\beta-2}\sum_{j,l\ge 0}\frac{\d^2\Omega}{\d v^{(j)}_i\d v^{(l)}_k}\left(v_{i+1}^{(j+1)}v_{k+1}^{(l+1)}-v_i^{(j+1)}v_{k+2}^{(l+1)}\right)\\
&+\lambda\left[\sum_{i=1}^{\beta-2}\sum_{j\ge 0}\frac{\d}{\d v_i^{(j)}}\left(\d_x^{2g+1}r_{\beta-2g}\right)v_{i+1}^{(j+1)}+\underline{\sum_{i=1}^{\beta-2}\sum_{j\ge 0}\frac{\d}{\d v_i^{(j)}}\left(v_{\beta-2g+1}^{(2g+1)}\right)\d_x^{j+1}r_i}\,\right]\\
&-\lambda\d_x^{2g+1}\left[\underline{\d_x r_{\beta-2g+1}}+\sum_{k=1}^{\beta-2g-1}\frac{\d r_{\beta-2g}}{\d v_k}v_{k+1}^{(1)}-v_1 v_{\beta-2g}^{(1)}\right]-\lambda\d_x\left(v_1 v_{\beta-2g}^{(2g+1)}\right)\in\cR_{v;2}
\end{align*}
and $D\in\cR_{v;\ge 3}$. Since~\eqref{eq:function of R} is equal to zero, we have $C=0$. The underlined terms cancel each other. Using the identity $\sum_{j\ge 0}\frac{\d(\d_x P)}{\d v_i^{(j)}}\d_x^j Q=\d_x\left(\sum_{j\ge 0}\frac{\d P}{\d v_i^{(j)}}\d_x^j Q\right)$, $P,Q\in\cR_v$, $i\ge 1$, we also compute
\begin{gather*}
\sum_{i=1}^{\beta-2}\sum_{j\ge 0}\frac{\d}{\d v_i^{(j)}}\left(\d_x^{2g+1}r_{\beta-2g}\right)v_{i+1}^{(j+1)}=\d_x^{2g+1}\left(\sum_{i=1}^{\beta-2}\frac{\d r_{\beta-2g}}{\d v_i}v_{i+1}^{(1)}\right).
\end{gather*}
As a result, 
\begin{gather*}
C=\sum_{i,k=1}^{\beta-2}\sum_{j,l\ge 0}\frac{\d^2\Omega}{\d v^{(j)}_i\d v^{(l)}_k}\left(v_{i+1}^{(j+1)}v_{k+1}^{(l+1)}-v_i^{(j+1)}v_{k+2}^{(l+1)}\right)+\lambda\left(\d_x^{2g+1}\left(v_1 v_{\beta-2g}^{(1)}\right)-\d_x\left(v_1 v_{\beta-2g}^{(2g+1)}\right)\right),
\end{gather*}
which, denoting $\gamma:=\beta-2g\ge 1$, we write as
\begin{align}
&\sum_{i=1}^{\gamma-2}\sum_{j=0}^{2g}\omega_{i,j}\left(v_{i+1}^{(j+1)}v_{\gamma-i}^{(2g-j+1)}-v_i^{(j+1)}v_{\gamma+1-i}^{(2g-j+1)}\right)+\lambda\left(\d_x^{2g+1}\left(v_1 v_{\gamma}^{(1)}\right)-\d_x\left(v_1 v_{\gamma}^{(2g+1)}\right)\right)=\notag\\
=&\sum_{i=1}^{\gamma-2}\sum_{j=0}^{2g}(\omega_{i,j}-\omega_{i+1,j})v_{i+1}^{(j+1)}v_{\gamma-i}^{(2g-j+1)}-\sum_{j=0}^{2g}\omega_{1,j}v_1^{(j+1)}v_{\gamma}^{(2g-j+1)}\notag\\
&+\lambda\left(\d_x^{2g+1}\left(v_1 v_{\gamma}^{(1)}\right)-\d_x\left(v_1 v_{\gamma}^{(2g+1)}\right)\right)=\notag\\
=&\frac{1}{2}\sum_{i=1}^{\gamma-2}\sum_{j=0}^{2g}(2\omega_{i,j}-\omega_{i+1,j}-\omega_{i-1,j})v_{i+1}^{(j+1)}v_{\gamma-i}^{(2g-j+1)}\label{eq:leading term,line 1}\\
&-\sum_{j=0}^{2g}\omega_{1,j}v_1^{(j+1)}v_{\gamma}^{(2g-j+1)}+\lambda\left(\d_x^{2g+1}\left(v_1 v_{\gamma}^{(1)}\right)-\d_x\left(v_1 v_{\gamma}^{(2g+1)}\right)\right),\label{eq:leading term,line 2}
\end{align}
where we adopt the convention $\omega_{i,j}:=0$ if $i\le 0$ or $i\ge \gamma-1$.\\

The expression in line~\eqref{eq:leading term,line 1} doesn't contain monomials of the form $v_1^{(i)}v_\gamma^{(j)}$ and, therefore, the expressions in lines~\eqref{eq:leading term,line 1} and~\eqref{eq:leading term,line 2} vanish:
\begin{align}
&2\omega_{i,j}-\omega_{i+1,j}-\omega_{i-1,j}=0,\qquad 1\le i\le \gamma-2,\quad 0\le j\le 2g, \label{eq:leading term,condition 1}\\
&\omega_{1,j}=
\begin{cases}
2g\lambda,&\text{if $j=0$},\\
{2g+1\choose j+1}\lambda,&\text{if $1\le j\le 2g$}.
\end{cases}\label{eq:leading term,condition 2}
\end{align}
If $\gamma=1$ or $\gamma=2$, then $\Omega=0$, and from~\eqref{eq:leading term,condition 2} we immediately get $\lambda=0$, which contradicts the assumption $\lambda\ne 0$. Suppose $\gamma\ge 3$. Solving relations~\eqref{eq:leading term,condition 1} step by step for $i=1,2,\ldots,\gamma-3$, we obtain $\omega_{i,j}=i\omega_{1,j}$ for $1\le i\le \gamma-2$. Then for $i=\gamma-2$ relation~\eqref{eq:leading term,condition 1} says that $0=2\omega_{\gamma-2,j}-\omega_{\gamma-3,j}=(\gamma-1)\omega_{1,j}$, which gives $\omega_{1,j}=0$ and hence all $\omega_{i,j}=0$. From relation~\eqref{eq:leading term,condition 2} we then obtain $\lambda=0$, which contradicts the assumption $\lambda\ne 0$.\\

{\it \underline{Case 2: $d\ge 2$}}. The expression~\eqref{eq:function of R} has the form $\eps^{2g}(C+D)+O(\eps^{2g+2})$, where 
\begin{gather}\label{eq:C-expression}
C=\sum_{i,k=1}^{\beta-2}\sum_{j,l\ge 0}\frac{\d^2 A}{\d v^{(j)}_i\d v^{(l)}_k}v_{i+1}^{(j+1)}v_{k+1}^{(l+1)}-\sum_{k=1}^{\beta-2}\sum_{l\ge 0}\d_x\left(\frac{\d A}{\d v^{(l)}_k}\right)v_{k+2}^{(l+1)}\in \cR_{v;d},
\end{gather}
and $D\in\cR_{v;\ge d+1}$. Since~\eqref{eq:function of R} is equal to zero, we have $C=0$. Let $k_0$ be the largest $k$ such that $\frac{\d A}{\d v_k^{(l)}}\ne 0$ for some $l=l_0$. Then from~\eqref{eq:C-expression} it is clear that 
$$
\frac{\d C}{\d v_{k_0+2}^{(l_0+1)}}=-\d_x\frac{\d A}{\d v_{k_0}^{(l_0)}}\ne 0,
$$
which contradicts the fact that $C=0$.
\end{proof}

\end{document}